\documentclass[a4paper,10pt]{amsart}
\usepackage{amssymb,amsmath,amsthm,hyperref,verbatim}
\usepackage{diagrams}
\usepackage{enumerate}

\DeclareMathOperator{\pr}{pr}  

\newcommand{\abar}{{\ensuremath{\bar{a}}}}
\newcommand{\bbar}{{\ensuremath{\bar{b}}}}
\newcommand{\cbar}{{\ensuremath{\bar{c}}}}
\newcommand{\dbar}{{\ensuremath{\bar{d}}}}
\newcommand{\pbar}{{\ensuremath{\bar{p}}}}
\newcommand{\kbar}{{\ensuremath{\bar{k}}}}
\newcommand{\xbar}{{\ensuremath{\bar{x}}}}
\newcommand{\ybar}{{\ensuremath{\bar{y}}}}
\newcommand{\zbar}{{\ensuremath{\bar{z}}}}
\newcommand{\mbar}{{\ensuremath{\bar{m}}}}
\newcommand{\Xbar}{{\ensuremath{\bar{X}}}}

\newcommand{\fbar}{{\ensuremath{\bar{f}}}}

\DeclareMathOperator{\Th}{Th}  
\DeclareMathOperator{\tp}{tp}  

\DeclareMathOperator{\rk}{rk}

\DeclareMathOperator{\td}{td}  

\DeclareMathOperator{\loc}{Loc}   

\newcommand{\alg}{\ensuremath{\mathrm{alg}}} 

\DeclareMathOperator{\ldim}{ldim}  
\DeclareMathOperator{\mrk}{mrk}  
\DeclareMathOperator{\Mat}{Mat}  

\DeclareMathOperator{\etd}{etd}  

\DeclareMathOperator{\ecl}{ecl} 

\newcommand{\restrict}[1]{\ensuremath{\!\!\upharpoonright_{#1}}}

\newcommand{\N}{\ensuremath{\mathbb{N}}}
\newcommand{\Z}{\ensuremath{\mathbb{Z}}}
\newcommand{\Q}{\ensuremath{\mathbb{Q}}}

\newcommand{\Cexp}{\ensuremath{\mathbb{C}_{\mathrm{exp}}}}

\newcommand{\Loo}{\ensuremath{L_{\omega_1,\omega}}}

\newcommand{\ga}{\ensuremath{\mathbb{G}_\mathrm{a}}}   
\newcommand{\gm}{\ensuremath{\mathbb{G}_\mathrm{m}}}  

\renewcommand{\phi}{\varphi}
\renewcommand{\le}{\ensuremath{\leqslant}}
\renewcommand{\ge}{\ensuremath{\geqslant}}
\newcommand{\tuple}[1]{\ensuremath{\langle #1 \rangle}}
\newcommand{\class}[2]{\ensuremath{\left\{ #1 \,\left|\, #2 \right.\right\}}}

\newcommand{\iso}{\cong}

\newcommand{\into}{\hookrightarrow}

\newcommand{\subs}{\subseteq} 

\newcommand{\elsubs}{\preccurlyeq} 

\newcommand{\minus}{\ensuremath{\smallsetminus}}

\newcommand{\strong}{\ensuremath{\lhd}} 
\newcommand{\nstrong}{\ensuremath{\not\kern-4pt\lhd\;}} 

\newcommand{\gen}[1]{\ensuremath{\left\langle #1 \right\rangle}} 

\newcommand{\cross}{\ensuremath{\times}}

\newbox\noforkbox \newdimen\forklinewidth
\forklinewidth=0.3pt \setbox0\hbox{$\textstyle\smile$}
\setbox1\hbox to \wd0{\hfil\vrule width \forklinewidth depth-2pt
  height 10pt \hfil}
\wd1=0 cm \setbox\noforkbox\hbox{\lower 2pt\box1\lower
2pt\box0\relax}
\def\unionstick{\mathop{\copy\noforkbox}\limits}

\def\nonfork_#1{\unionstick_{\textstyle #1}}

\setbox0\hbox{$\textstyle\smile$} \setbox1\hbox to \wd0{\hfil{\sl
/\/}\hfil} \setbox2\hbox to \wd0{\hfil\vrule height 10pt depth
-2pt width
                \forklinewidth\hfil}
\newbox\doesforkbox
\setbox\doesforkbox\hbox{\lower 2pt\box1 \lower
2pt\box2\lower2pt\box0\relax}
\def\nunionstick{\mathop{\copy\doesforkbox}\limits}

\def\fork_#1{\nunionstick_{\textstyle #1}}

\newcommand{\findep}[4]{\ensuremath{#1 \nonfork_{#3}^{#4} #2}}

\newcommand{\algindep}[3]{\findep{#1}{#2}{#3}{\mathrm{ACF}}}

\newcommand{\ra}[3]{\ensuremath{#1 \stackrel{#2}{\longrightarrow} #3}}

\newcommand{\map}[5]{
\begin{eqnarray*}
#1 & \stackrel{#2}{\longrightarrow} & #3\\ #4 & \longmapsto & #5
\end{eqnarray*}}

\newcommand{\leteq}{\mathrel{\mathop:}=}

\newarrowmiddle{normal}\lhtriangle\rttriangle\uhtriangle\dhtriangle

\newarrowmiddle{iso}{\cong}{\cong}{}{}   

\newarrow{Partial}{-}{-}{-}{-}{harpoon}     

\newarrow{Equal}=====     

\newarrow{Strong}{hooka}-{normal}->     

\newarrow{Str}{}{}{normal}{}{}     

\newarrow{Iso}--{iso}->     

\newarrow{NT}====>     

\newcommand{\seacness}{strong exponential-algebraic closedness}
\newcommand{\vfk}{very full kernel}
\newcommand{\Qalg}{\ensuremath{\Q^{\mathrm{alg}}}}

\newcommand{\sstrong}{\ensuremath{\mathrel{\prec \kern -0.53ex \shortmid}}}

\newcommand{\B}{\ensuremath{\mathbb{B}}}  
\newcommand{\A}{\ensuremath{\mathbb{A}}}  

\newcommand{\wbar}{\bar{w}}

\newcommand{\ECF}{\ensuremath{\mathbf{ECF}}} 
\newcommand{\ECFSK}{\ensuremath{\mathbf{ECF_{SK}}}}  
\newcommand{\ECFSKCCP}{\ensuremath{\mathbf{ECF_{SK,CCP}}}}  
\newcommand{\SPOK}{Schanuel property over the kernel}

\newtheorem{prop}{Proposition}[section]
\newtheorem{cor}[prop]{Corollary}
\newtheorem{theorem}[prop]{Theorem}
\newtheorem{lemma}[prop]{Lemma}

\newtheorem{fact}[prop]{Fact}

\newtheorem{conj}[prop]{Conjecture}

\theoremstyle{definition}
\newtheorem{defn}[prop]{Definition}

\title[Exponentially closed fields]{Exponentially closed fields and the conjecture on intersections with tori}
\author{Jonathan Kirby}
\address{Jonathan Kirby (corresponding author)\\
School of Mathematics\\
University of East Anglia\\
Norwich Research Park\\
Norwich\\
NR4~7TJ}
\email{jonathan.kirby@uea.ac.uk}

\author{Boris Zilber}
\address{Boris Zilber\\
Mathematical Institute\\
Univesity of Oxford\\
24--29 St Giles'\\
Oxford\\
OX1~3LB}
\email{zilber@maths.ox.ac.uk}

\date{version 2.2, \today}

\subjclass[2000]{03C65, 11G35}

\keywords{Exponential fields, anomalous intersections, Schanuel's conjecture, predimension}

\begin{document}

\begin{abstract}
{We give an axiomatization of the class \ECF\ of exponentially closed fields, which includes the pseudo-exponential fields previously introduced by the second author, and show that it is superstable over its interpretation of arithmetic. Furthermore, 
\ECF\ is exactly the elementary class of the pseudo-exponential fields if and only if the diophantine conjecture CIT on atypical intersections of tori with subvarieties is true.}
\end{abstract}
\maketitle


\section{Introduction}

\subsection{Pseudo-exponential fields}

In \cite{Zilber05peACF0}, the second author introduced a class of exponential fields he called \emph{pseudo-exponential fields}, as the class of models $\tuple{F;+,\cdot,\exp}$ 
of the following five axioms, including the statement of the well-known Schanuel conjecture. 
\begin{description}
 \item[1. ELA-field] $F$ is an algebraically closed field of characteristic zero, and its exponential map $\exp$ is a homomorphism from its additive group to its multiplicative group, which is surjective.

 \item[2. Standard kernel] the kernel of the exponential map is an infinite cyclic group generated by a transcendental element $\tau$.

 \item[3. Schanuel Property] The \emph{predimension function} 
\[\delta(\xbar) \leteq \td(\xbar, \exp(\xbar))- \ldim_\Q(\xbar)\]
satisfies $\delta(\xbar) \ge 0$ for all tuples $\xbar$ from $F$. 
\item[4. Strong exponential-algebraic closedness] If $V$ is a rotund, additively and multiplicatively free subvariety of $\ga^n\cross \gm^n$ defined over $F$ and of dimension $n$, and $\abar$ is a finite tuple from $F$, then there is $\xbar$ in $F$ such that $(\xbar,e^\xbar) \in V$ and is generic in $V$ over $\abar$.

\item[5. Countable Closure Property] For each finite subset $X$ of $F$, the exponential algebraic closure $\ecl^F(X)$ of $X$ in $F$ is countable. 
\end{description}
We call any model of axiom 1 an \emph{ELA-field}: E for exponentiation, L for the surjectivity (every non-zero element has a logarithm) and A for algebraically closed. 
Precise definitions of the terms in axioms 4 and 5 are given later, in sections \ref{classification of strong extensions} and \ref{exp trans section} respectively. Intuitively, axiom 4 says that any system of equations which can have a solution in some suitable exponential extension field  does already have a solution in $F$. It is the analogue for exponential fields of the algebraic closedness axiom for fields, which says that any polynomial equation (or, equivalently, any system of polynomial equations) which can have a solution in some extension field already has a solution in $F$. Axiom 5 says that such a system should only have countably many solutions. 

We denote by \ECFSK\ the class of models of axioms 1---4, and call the models \emph{Exponentially-Closed Fields with Standard Kernel}. We also denote by \ECFSKCCP\ the class of models of axioms 1---5. (In \cite{Zilber05peACF0} the same classes were called $\mathcal{EC}^*_{st}$ and $\mathcal{EC}^*_{st,ccp}$.) 

The main theorem of \cite{Zilber05peACF0} was that \ECFSKCCP\  has exactly one model in each uncountable cardinality, up to isomorphism. This categoricity theorem was not proved entirely in one paper. The proof depends on the main result from \cite{Zilber06covers}, and corrections to the two papers appeared in \cite{BZ11} and \cite{BK13patch}. The proof also uses the model-theoretic technique of \emph{quasiminimal excellent classes} which were developed for this purpose in \cite{Zilber05qmec}, and further developed and simplified in \cite{OQMEC} and \cite{BHHKK}. 

We write \B\ for the model of $\ECFSK$ of cardinality $2^{\aleph_0}$. The complex exponential field $\Cexp$ is known to satisfy axioms 1 and 2 (trivially) and 5 (less trivially). Schanuel's conjecture is a fundamental conjecture of transcendental number theory. Since strong exponential-algebraic closedness  is very natural from the model-theoretic point of view, it makes sense to conjecture that it holds for $\Cexp$. Together, these  two conjectures are therefore equivalent to the assertion that $\Cexp$ is isomorphic to $\B$. 

The axioms for $\ECFSK$  are not all first-order expressible, but they can all be expressed in the logic $\Loo$, which allows countable conjunctions of formulas. In fact, \ECFSK\ can also be viewed as the class of models of a complete first-order theory which omit the type of a non-standard integer \cite{NAPE}. We denote the first-order theory by $T_\B$. (The countable closure property is not expressible in $\Loo$, but can be expressed using the quantifier $Q x$: \emph{there exist uncountably many $x$ such that \ldots}. However, we will make no use of that axiom for the rest of the paper.)

While \ECFSK\ is a well-behaved class of structures, it is not an elementary class. Understanding the first-order theory $T_\B$ should give much more information which we believe may be useful in understanding the analytic geometry of $\Cexp$, perhaps even without assuming that $\Cexp$ is isomorphic to \B. This paper seeks to give an understanding of $T_\B$. Using the categoricity theorem for \ECFSKCCP\ we see that $T_\B$ is the complete first-order theory of $\B$, hence our notation. However, the results of this paper do not depend on the categoricity theorem. 

An obvious obstacle to the goal of understanding $T_\B$ comes from the integers. For any exponential field $F$, we will write $\ker(F) = \class{x\in F}{\exp(x) = 1}$, the kernel of the exponential map. We also define 
\[Z(F) = \class{r \in F}{\forall x[\exp(x) = 1 \to \exp(rx)=1]},\]
the \emph{multiplicative stabilizer} of the kernel. In any exponential field $F$ with standard kernel, $Z(F)$ will actually be $\Z$, the standard integers, and thus the first-order theory $T_\B$ contains the theory of true arithmetic, and hence is undecidable and unstable. Nonetheless, we can still look for stable-like behaviour of the theory. One analogue is the theory ACVF of algebraically closed valued fields. There the value group is unstable, but the stable part of the theory can be understood separately from the value group in terms of stable domination. The situation here is more complicated in some ways, because the value group in ACVF is still fairly tame, being o-minimal, whereas arithmetic is not tame in any sense. However, arithmetic can be isolated from the rest of the theory in a fairly strong way if and only if the diophantine conjecture known as CIT (Conjecture on Intersections of Tori with subvarieties) is true. We next describe our setting of Exponentially Closed Fields, and then explain how CIT is involved.
 
\subsection{Exponentially Closed Fields}

In this paper we introduce and study the class of Exponentially-Closed Fields which we denote by \ECF, which is obtained by slightly weakening the axioms for \ECFSK. Axioms 1 and 4 are the same, but 2 and 3 are replaced by 2$'$ and 3$'$ below.
\begin{description}
 \item[2$'$a. Cyclic kernel] The kernel is a cyclic $Z(F)$-module.
 \item[2$'$b. Transcendental kernel] Every element of the kernel is transcendental over $Z(F)$.
 \item[2$'$c. Theory of true arithmetic] $Z(F)$ with the restrictions of $+$ and $\cdot$ is a model of the full first-order theory of $\tuple{\Z;+;\cdot}$.
 \item[3$'$. The Schanuel Property over the kernel (SPOK)] \ The predimension function 
 \[ \Delta_F(\xbar) \leteq \td(\xbar, \exp(\xbar) /  \ker(F)) - \ldim_\Q(\xbar /  \ker(F)) \]
 satisfies $\Delta_F(\xbar) \ge 0$ for all tuples $\xbar$ from $F$. 
 \end{description}
 By $\td(Y/X)$ we mean the transcendence degree of the field extension $\Q(XY)/\Q(X)$ and by $\ldim_\Q(Y/X)$ we mean the dimension of the $\Q$-vector space spanned by $X \cup Y$, quotiented by the subspace spanned by $X$.

The main unconditional result of this paper about \ECF\ is one of existence and uniqueness of saturated models in all cardinalities above the continuum. This cannot be literally true because of the presence of arithmetic in the kernel, but is true over the (sufficiently saturated) kernel.
\begin{theorem}\label{unique sat model theorem}
For each $\aleph_0$-saturated model $R$ of $\Th(\Z)$, and for each cardinal $\lambda > 2^{\aleph_0}$ with $\lambda \ge |R|$, there is exactly one model $M \in \ECF$ such that $Z(M) = R$ and such that $M$ is saturated over its kernel.
\end{theorem}
\emph{Saturation over the kernel} means saturated with respect to those extensions which are within the class \ECF\ and do not extend the kernel. The precise definition is~\ref{sat over kernel defn}. It follows fro m this theorem that the subclasses $\ECF_R$ of $\ECF$ where we fix the model of the integers to be $R$, are \emph{superstable homogeneous classes} in the sense of \cite{HySh:629}. One can therefore do a certain amount of stability theory in $\ECF$, provided one avoids extending the kernel.

\subsection{The Conjecture on Intersections with Tori}

CIT was first formulated by the second author in model-theoretic studies of exponentiation \cite{Zilber02esesc} and then independently rediscovered by Bombieri, Masser and Zannier \cite{BMZ07} and in a more general form by Pink \cite{Pink05}, \cite{Pink05b}. The conjecture generalises the well-known Mordell-Lang and Andr\'e-Oort conjectures. It is now an active field of research which applies methods of diophantine and algebraic geometry and the theory of Shimura varieties as well as model theory.

Let $W \subs \gm^n$ be an irreducible subvariety, defined over $\Qalg$. Let $H \subs \gm^n$ be an algebraic subgroup.

Let $X$ be a connected component of $W \cap H$. Then $X$ is said to be an \emph{atypical} component of the intersection if and only if
\[\dim X > \dim W + \dim H - n.\]
(Sometimes $X$ is called an \emph{anomalous} or \emph{unlikely} component of the intersection.)

Let $W^{atyp}$ be the union of all the atypical components of $W \cap H$ for all algebraic subgroups $H$ of $\gm^n$.

The Conjecture on Intersections of subvarieties with Tori states:
\begin{conj}
For any $n\in \N$, if $W \subs \gm^n$ is any irreducible subvariety then $W^{atyp}$ is a proper Zariski-closed subset of $W$.
\end{conj}

This statement of the conjecture is in the style of the presentation in section~5 of \cite{BMZ07} and is convenient for us. It is equivalent to the statement of Conjecture~1 in \cite{Zilber02esesc}.

\subsection{\ECF\ and CIT}

The axioms for \ECF\ are easily seen to be \Loo-expressible, and axioms 1 and 2$'$ are first-order, but it is not immediately clear whether  axioms 3$'$ and 4 are first-order. The notions of rotundity and multiplicative freeness of a subvariety which appear in axiom 4 are unconditionally first-order definable in the field language (but additive freeness is not) \cite[Theorem~3.2]{Zilber05peACF0}. Axiom 4 as a whole is first-order expressible assuming standard kernel \cite[Proposition~2.3]{NAPE}, but the proof there relies on being able to quantify over the standard integers, and does not generalise to \ECF.

We will show that axiom 3$'$ (the Schanuel Property over the kernel) is first-order axiomatizable if and only if CIT is true. In fact more is true.
\begin{theorem}\label{CIT implies ECF is first-order}
If CIT is true, then the axioms defining $\ECF$ are first-order expressible, and furthermore they axiomatise the complete theory $T_\B$. 
\end{theorem}
\begin{theorem}\label{CIT false implies ECF not fo}
If CIT is false, then \ECF\ is not an elementary class.
\end{theorem}
Axiom 4 is a technical strengthening of a more natural and simpler exponential-algebraic closedness property. Under CIT, this strengthening is not necessary.
\begin{theorem}\label{EAC theorem}
Assuming CIT, a model $F$ of axioms 1, 2$'$ and 3$'$ also satisfies axiom 4 (\seacness) if and only if it satisfies the following simpler version:
\begin{description}
\item[4$'$. Exponential-algebraic closedness] If $V$ is a rotund subvariety of $\ga^n\cross \gm^n$ defined over $F$ then there is $\xbar$ in $F$ such that $(\xbar,e^\xbar) \in V$.
\end{description}
\end{theorem}

In summary, if CIT is true then the effects of arithmetic can be contained within the kernel, and the theory $T_\B$ is otherwise tame, in fact \emph{superstable over the kernel} in a certain sense. However, if CIT is false then the effects of arithmetic will be seen in the first-order theory outside the kernel, the Schanuel property fails gravely in some models of the theory, the natural predimension notion loses its meaning, and any reasonable first-order stability seems to be impossible. Theorems~\ref{CIT implies ECF is first-order} and~\ref{CIT false implies ECF not fo} thus give us  a statement about exponentiation which is a highly non-obvious reformulation of CIT.
The more general form of CIT mentioned above is now known as the Zilber-Pink conjecture, and rather than concerning just multiplicative tori it deals with algebraic subgroups of semiabelian varieties and special subvarieties of Shimura varieties. We hope that ongoing work developing the ideas explained here will shed further light on the Zilber-Pink conjecture more broadly.

\subsection{Outline of the paper}
In section 2, we explain why the Schanuel property is not first-order expressible, then introduce the \SPOK\ and show that it is first-order expressible if and only if CIT is true. This proves Theorem~\ref{CIT false implies ECF not fo}.

In section 3, we recall many concepts about the algebra of exponential fields which were developed by the first author in \cite{FPEF}. In that paper only exponential fields with standard kernel were considered, and we also develop the theory as needed in the broader context of this paper. The notions of exponential algebraicity and exponential transcendence are also recalled from \cite{EAEF}. At the end of the section we explain all the terms needed for the statement of axiom 4, \seacness, and explain why it is the analogue for exponential fields of algebraic closedness for fields.

We introduce the concept of saturation over the kernel within the class \ECF\ in  section~4, and prove Theorem~\ref{unique sat model theorem}. Superstability over the kernel for the class \ECF\ follows. In section~5 we show that axiom 4 is first-order expressible assuming CIT (and under the other axioms) and thus the class \ECF\ is elementary assuming CIT. Again under CIT we prove completeness of the first-order theory, and the equivalence of axioms 4 and 4$'$, completing the proofs of Theorems~~\ref{CIT implies ECF is first-order} and~\ref{EAC theorem}. In section~6 we give some corollaries, and section~7 contains some final comments.

\subsection{Acknowledgement}

We would like to thank the anonymous referee who read the paper carefully and made many suggestions which greatly improved the presentation.

\section{The Schanuel property and the kernel}

There are two ways in which the Schanuel property might fail in an elementary extension of $\B$: inside the kernel or over the kernel. The first is mild and, as we will see, unavoidable. The second is a severe failing but happens only if CIT fails.

\subsection{Failure of the Schanuel property}

It is clear that axioms 1 and 2$'$ are expressible as first-order axioms schemes. However axiom 3, the Schanuel property, is not. To see this, suppose $F$ is an elementary extension of $\B$ with non-standard kernel. Then there is $r \in Z(F)$ which is transcendental, and since $Z(F)$ is a subring of $F$ we have $r^n \in Z(F)$ for all $n \in \N$. For any $t \in \ker(F)$, the set $\class{r^nt}{n \in \N}$ lies in $\ker(F)$ and is $\Q$-linearly independent, since $r$ is transcendental. The Schanuel property would imply that $\td(rt,r^2 t, \ldots,r^nt, 1,\ldots,1) = n$ for any $n \in \N$, but clearly this transcendence degree is at most 2. Hence the Schanuel property fails in all elementary extensions of \B\ with non-standard kernel.

\subsection{The \SPOK\ (SPOK)}

The above failure of the Schanuel property is inside the kernel. Axiom 3$'$, the \SPOK, essentially asserts that this is the only place where the Schanuel property can fail. We expressed the Schanuel property as the non-negativity of a predimension function $\delta$. The \SPOK\ can be expressed in the same form. Given an exponential field $F$, and a finite tuple $\xbar$ from $F$, define
\[\Delta_F(\xbar) \leteq \td(\xbar, \exp(\xbar) /  \ker(F)) - \ldim_\Q(\xbar /  \ker(F)). \]

Axiom 3$'$ asserts that $\Delta_F(\xbar) \ge 0$ for all finite tuples $\xbar$ from $F$.

If $F$ is algebraically closed then its multiplicative group $\gm(F) = \tuple{F^\times ; \cdot}$ is divisible and its torsion is the subgroup of roots of unity, which we denote by $\sqrt1$. Thus the quotient group $\gm(F)/\sqrt1$ is a $\Q$-vector space. If $A \subs \gm(F)$ then the $\Q$-linear dimension of its image $A /\sqrt1$ is called is \emph{multiplicative rank} and we denote it by $\mrk(A)$. Note that $e^a \in \sqrt1$ if and only if $a$ lies in the $\Q$-linear span of $\ker(F)$ and hence for any $\xbar$ in $F$ we have $\ldim_\Q(\xbar/\ker(F)) = \mrk(\exp(\xbar))$. Thus $\Delta_F(\xbar) = \td(\xbar,\exp(\xbar)/\ker(F)) - \mrk(\exp(\xbar))$.

Axiom 3$'$ can also be stated without using a predimension function.
\begin{lemma}
In any exponential field $F$, axiom 3$'$ holds if and only if for any $n$-tuple $\xbar \in F^n$, if $\td(\xbar,\exp(\xbar)/\ker(F)) < n$ then $\exp(\xbar)$ lies in a proper algebraic subgroup of $\gm^n(F)$.
\end{lemma}
\begin{proof}
A proper algebraic subgroup $H \subs \gm^n$ is given by a list of equations of the form 
\[ \hspace*{10em}    \prod_{i=1}^n y_i^{m_i} = 1 \hspace*{10em} (*)\]
where $m_i \in \Z$, not all zero. The left to right implication follows immediately and the right to left implication follows by induction on $n$.
\end{proof}

\subsection{Atypical intersections and the CIT}
We will show that the \SPOK\ is first-order expressible under the diophantine conjecture CIT. 

We may write the equations $(*)$ for a proper algebraic subgroup $H \subs \gm^n$ in matrix form as $\ybar^M = 1$.
We can define the \emph{depth} of $H$ to be the least $N \in \N$ such that $H$ is contained in some codimension 1 subgroup of $\gm^n$, which is given by a single equation of the form $(*)$ in which every $m_i$ satisfies $|m_i| \le N$. Then CIT can equivalently be stated as: for every irreducible subvariety $W \subs \gm^n$, defined over $\Q$, there is $N \in \N$ such that $W^{atyp}$ is contained in the union of the set of proper algebraic subgroups of $\gm^n$ of depth at most $N$.

Note that $W^{atyp}$ is certainly \Loo-definable (uniformly in parameters), but if CIT is true then the above shows that $W^{atyp}$ is first-order definable (even uniformly in parameters).

\subsection{Generic fibres}
We write $G$ for the algebraic group $\ga \cross \gm$. There is the natural projection map $\pr: G^n \to \gm^n$. Given $V \subs G^n$ irreducible and $b \in \pr(V)$, we consider the fibre $V(b) = \class{v\in V}{\pr(v) = b}$. Define 
\[V^{gf} = \class{v\in V}{\dim V(\pr(v)) = \dim V - \dim \pr(V)}\]
which is the union of the fibres of generic dimension. If $V$ is reducible, we define $V^{gf}$ to be the union of the $W^{gf}$ as $W$ runs through the irreducible components of $V$. The fibre dimension theorem (see for example section~2 of \cite{BMZ07} for a clear statement) tells us that $V^{gf}$ is constructible and dense in $V$, and if $V$ varies in a parametric family then $V^{gf}$ is definable uniformly in the parameters.

\subsection{Axiomatizing the \SPOK}
 For a subvariety $\mathcal{V} \subs G^n \cross \A^m$, where $\A$ is affine space, and $\pbar \in \A^m$, we write $V_\pbar \subs G^n$ for the fibre of $\mathcal{V}$ at $\pbar$, and consider $\mathcal{V} = (V_\pbar)_{\pbar \in \A^m}$ as a parametric family of subvarieties of $G^n$. For such a family, consider the axiom: 
\[(\forall \bar{k} \in \ker^m)(\forall \bar{x})\left[\dim V_\kbar \ge n \vee (\xbar,e^\xbar) \notin V_\kbar^{gf} \vee e^\xbar \in \pr(V_\kbar)^{atyp} \right]\]

We take the SPOK scheme to be the scheme of all such axioms, for all $n,m \in \N$, and for all families $\mathcal{V}$ defined over $\Q$. So every subvariety of $G^n$ defined over $\ker(F)$ will appear as $V_\kbar$ for some family $\mathcal{V}$ and some tuple $\kbar$ from $\ker(F)$. By the fibre dimension theorem, the SPOK scheme is expressible as an \Loo-sentence, and, if CIT is true, it is expressible as a first-order axiom scheme.

\begin{prop}\label{strong kernel is fo}
 Let $F$ be an exponential field. Then $F$ has the \SPOK\ if and only if it satisfies the SPOK scheme.
\end{prop}
\begin{proof}
 Suppose the scheme holds, and $\xbar \in F^n$ with $\td(\xbar,e^\xbar/\ker(F)) < n$. Let $V = \loc((\xbar,e^\xbar)/\ker(F))$. Then $\dim V < n$, so by the scheme, either $e^\xbar \in \pr(V)^{atyp}$ or $(\xbar,e^\xbar) \notin V^{gf}$. The latter is impossible, since $V \minus V^{gf}$ is constructible, defined over $\ker(F)$, and not dense in $V$, but $V$ is the locus of $(\xbar,e^\xbar)$ over $\ker(F)$. Hence $e^\xbar \in  \pr(V)^{atyp}$, which means that $e^\xbar$ lies in a proper algebraic subgroup of $\gm^n$, so $F$ has the \SPOK.

Conversely, suppose $F$ has the \SPOK, that $V \subs G^n$ is defined over $\ker(F)$ with $\dim V < n$, and that $(\xbar,e^\xbar) \in V^{gf}$. We must show $e^\xbar \in \pr(V)^{atyp}$. Let $H$ be the smallest algebraic subgroup of $\gm^n$ such that $e^\xbar \in H$. 
By the \SPOK\ we have
\[0 \le \Delta_F(\xbar) =  \td(\xbar,e^\xbar/\ker(F)) - \dim H.\]
Also $ \td(\xbar,e^\xbar/\ker(F)) = \td(\xbar/e^\xbar,\ker(F)) + \td(e^\xbar/\ker(F))$.

Let $X$ be the component of $\pr(V) \cap H$ containing $e^\xbar$, so we have $\td(e^\xbar/\ker(F)) \le \dim X$. Also $(\xbar,e^\xbar) \in V^{gf}$, so 
\[\td(\xbar/e^\xbar,\ker(F)) \le \dim V - \dim \pr(V) < n - \dim \pr(V).\]
Hence, putting these together, 
\[n - \dim \pr(V) + \dim (X) > \dim H\]
which means that $X$ is an atypical component of the intersection, and hence $e^\xbar \in \pr(V)^{atyp}$ as required.
\end{proof}

\subsection{Consequence of the failure of CIT}
We have seen that if CIT is true then the strong kernel property is first-order axiomatizable. Now we show the converse.

\begin{prop}\label{false CIT}
Suppose that CIT is false. Then no $F \in \ECF$ is $\aleph_0$-saturated.
\end{prop}
Theorem~\ref{CIT false implies ECF not fo} follows at once.
\begin{proof}
Assuming CIT is false, we will show that for any $F \in \ECF$ there is a type over a finite subset of $F$ such that any elementary extension of $F$ realising the type does not satisfy the \SPOK.

In \cite{Zilber02esesc} it is shown that if CIT is false then there is a counterexample defined over $\Q$. So suppose that $W \subs \gm^n$ is a counterexample to CIT, defined over $\Q$. So for each $N \in \N^+$, there is $H_N$, an algebraic subgroup of $\gm^n$ of depth strictly greater than $N$ such that 
\[ \dim(W \cap H_N) > \dim W + \dim H_N - n.\]
There are only finitely many possible values for $\dim H_N$ and for $\dim(W \cap H_N)$, so by passing to a subsequence of the $H_N$ we may assume for some $h$ and $t$ that, for all $N$, $\dim H_N = h$ and $\dim(W \cap H_N) = t$. 

We assume that $F$ has non-standard kernel. (If not, replace $F$ by an elementary extension.) Choose $r_{ij} \in Z(F)$, algebraically independent over $\Q$, for $i=1,\ldots,n, j=1,\ldots,t$. Let $X$ be the intersection of the generic hypersurfaces in $\gm^n$ given by the equations
$\sum_{i=1}^n r_{ij}y_i = 1$
for $j=1,\ldots,t$, and let $W' = W \cap X$. Then $\dim W' = \dim W - t$, and $\dim (W' \cap H_N) = \dim (W \cap H_N) - t = 0$ for each $N$. In particular, $W' \cap H_N$ is non-empty. Thus we have
\[h+ \dim W' < n.\]

Consider the following formulas $\phi_N(\xbar,\mbar)$ in free variables $x_1,\ldots,x_n$ and $(m_{ij})_{i,j = 1}^n$:
\begin{multline*}
\bigwedge_{i,j=1}^n m_{ij} \in Q \wedge M\xbar \in \ker^n \wedge \rk(M) = n-h \wedge e^\xbar \in W' \\
\wedge \bigwedge_{\bar{\mu} \in \{-N,\ldots,N\}^n \minus \{\bar{0}\}} \prod_{i=1}^n e^{\mu_i x_i} \neq 1
\end{multline*}
where $Q$ is the field of fractions of $Z$, and $M =(m_{ij})_{i,j = 1}^n$ considered as a matrix.

We show that each $\phi_N(\xbar,\mbar)$ is satisfiable in $F$. So take $H_N$ as above, and $\bbar \in (W \cap H_N)(F)$. Choose $M \in \Mat_{n\times n}(\Q)$ such that $H$ is given by $\ybar^M = 1$ and choose $\abar \in F^n$ such that $e^\abar = \bbar$, which is possible since $F$ is an ELA-field. Then $e^{M\abar} = \bar{1}$, so $M\abar \in \ker(F)^n$. Also $\rk(M) = n-h$ and $e^\abar \in W'$, but since $\bbar$ is generic in $W' \cap H_N$ (every point is generic because the dimension is 0) and the depth of $H_N$ is greater than $N$, we also have $\prod_{i=1}^n e^{\mu_i a_i} \neq 1$ whenever $\bar{\mu} \in \{-N,\ldots,N\}^n \minus \{\bar{0}\}$.

So each $\phi_N(\xbar,\mbar)$ is satisfiable in $F$, but then the set of formulas $\class{\phi_N(\xbar,\mbar)}{N \in \N}$ is finitely satisfiable, and hence by compactness it is satisfied in some elementary extension $K$ of $F$. Say $K \models \phi_N(\abar, \mbar)$ for all $N \in \N$.

Then since $M\abar \in \ker(K)$ and $M$ is a matrix over $Q(K)$ of rank $n-h$, we have $\ldim_{Q(K)}(\abar/\ker(K)) \le h$. Thus:
\begin{eqnarray*}
\td(\abar,e^\abar/\ker(K)) & \le & \td(\abar/\ker(K)) + \td(e^\abar/\ker(K)) \\
 & \le & \ldim_{Q(K)}(\abar/\ker(K)) + \td(e^\abar/\ker(K)) \\
& \le & h + \dim W' \\
& < & n.
\end{eqnarray*}
However, no nontrivial equations $\prod_{i=1}^n e^{\mu_ia_i} = 1$ hold where the $\mu_i \in \Z$, hence $\mrk(e^\abar) = n$, so $\ldim_\Q(\abar/\ker(K)) = n$, and so 
\[\Delta_K(\abar) = \td(\abar,e^\abar/\ker(K)) - \ldim_\Q(\abar/\ker(K)) < 0.\]
Thus $K$ does not have the \SPOK.
\end{proof}

\section{Some algebra of exponential fields}

The paper \cite{FPEF} of the first author develops the algebra of exponential fields and their extensions, under the restriction that the kernel is countable and does not extend. In this section we extend that work to consider also extensions where the kernel is uncountable and can extend.

For the proofs it is necessary to deal with partial exponential fields, where the exponential map is only partially defined. We start by considering them and properties of the kernel. \emph{Strong extensions} play a very important role in the earlier work and for the case where the kernel is not fixed we introduce the more general notion of \emph{semistrong} extensions. 
We then give two technical but essential results concerning the uniqueness of free ELA-extensions of a partial exponential field, analogous to the main technical result from \cite{FPEF}, following which we recall the notions of exponential algebraicity and exponential transcendence from another paper \cite{EAEF}. Finally we give a classification of the finitely generated kernel-preserving strong extensions of ELA-fields, which also serves to explain axiom 4, \seacness, and to define all the relevant terminology.

\begin{defn}\label{partial E-field defn}
  A \emph{partial E-field} is a two-sorted structure 
\[\tuple{F,D(F);+,\cdot,+_D,(q\cdot)_{q \in \Q},\alpha,\exp_F}\]
 where $\tuple{F;+,\cdot}$ is a field of characteristic zero, $\tuple{D(F);+_D,(q\cdot)_{q \in \Q}}$ is a \Q-vector space, $\ra{\tuple{D(F);+_D}}{\alpha}{\tuple{F;+}}$ is an injective homomorphism of additive groups, and $\ra{\tuple{D(F);+_D}}{\exp_F}{\tuple{F;\cdot}}$ is a homomorphism.
\end{defn}
We identify $D(F)$ with $\alpha(D(F))$, and hence a (total) E-field is a partial E-field. We write $I(F) = \class{\exp_F(a)}{a \in D(F)}$ for the image of the exponential map of $F$.

\subsection{Full and very full kernels}

We say that a partial (or total) E-field has \emph{full kernel} if and only if the group $\sqrt{1}$ of all roots of unity is contained in $I(F)$. 

As abstract abelian groups, $D(F)$ and $I(F)$ are both divisible. $D(F)$ is torsion-free, and the torsion in $I(F)$ is equal to $\sqrt{1} \cap I(F)$, which is the image of the $\Q$-linear span of the kernel. Considering the torsion, we see that $F$ has full kernel if and only if for each $n \in \N^+$, the quotient $\ker(F) / n \ker(F)$ is cyclic of order $n$. Equivalently, $\tuple{\ker(F);+}$ is a model of the complete theory of $\tuple{\Z;+}$.

We now give some relevant properties of this theory. All references are to \cite{Rothmaler}. Any model $M$ of  $\Th\tuple{\Z;+}$ splits as a direct sum $M_r \oplus M_d$ where $M_d$ is the largest divisible subgroup of $M$ and $M_r$ is called a \emph{reduced part} of $M$. We have $M_r \elsubs M$ \cite[Lemma~15.5.5(a)]{Rothmaler}.

Let $\hat{\Z}$ be the profinite completion of $\tuple{\Z;+}$. Then $\hat{\Z} \models \Th\tuple{\Z;+}$, and if $M$ is any model of $\Th\tuple{\Z;+}$ then $M_r$ embeds elementarily into $\hat{\Z}$ [15.5.5(a) and~15.6.2(1)]. (In fact, in the cited reference, $\hat{\Z}$ is defined differently, but the equivalence of the definitions is given on page 265.)

An abelian group is said to be \emph{algebraically compact} if is is saturated with respect to systems of certain (positive primitive) formulas. A model $M$ of $\Th\tuple{\Z;+}$ is algebraically compact if and only if $M_r$ is algebraically compact [15.5.5(2)] if and only if $M_r \iso\hat{\Z}$ [15.6.2(3)]. Furthermore, if $M$ is $\aleph_0$-saturated then it is algebraically compact [15.5.3(1)].

\begin{defn}
 We say that $F$ has \emph{\vfk} if and only if $\ker(F)$ is an algebraically compact model of $\Th\tuple{\Z;+}$.
\end{defn}



\begin{defn}
A \emph{coherent system of roots of unity} is a sequence $(c_m)_{m\in\N^+}$ such that $c_1 = 1$ and for each $r,m \in \N^+$ we have $c_{rm}^r = c_m$.
\end{defn}

\begin{lemma}\label{vfk lemma}
A partial E-field $F$ has \vfk\ if and only if for any coherent system of roots of unity $(c_m)_{m\in\N^+}$ there is $a \in D(F)$ such that for each $m\in\N^+$ we have $\exp_F(a/m) = c_m$.
\end{lemma}
\begin{proof}
Write $\sqrt[m]1$ for the group of $m^\mathrm{th}$ roots of unity. Define a group homomorphism by
\map{\ker(F)}{\theta}{\prod_{m \in \N^+}\sqrt[m]{1}}{a}{(\exp_F(a/m))_{m\in\N^+}}
Recall that $\tau$ is the cyclic generator of the kernel, and so $e^{\tau/m}$ is a primitive $m^{\mathrm{th}}$ root of unity for each $m \in \N^+$. So we can define isomorphisms $\sqrt[m]1 \iso \Z/m\Z$ by $e^{\tau/m} \mapsto 1 + m\Z$, which together make an isomorphism 
\[\prod_{m \in \N^+}\sqrt[m]{1} \iso \prod_{m \in \N^+} \Z/m\Z\]
and we identify these groups along the isomorphism. Then for each $a \in \ker(F)$, $\theta(a)$ is a coherent system of roots of unity, which is exactly the same condition as being an element of $\hat{\Z}$. The kernel of $\theta$ is the divisible part of $\ker(F)$, so $\theta$ restricts to an embedding of the reduced part of $\ker(F)$ into $\hat{\Z}$, which has no proper self-embeddings. Thus the image of $\theta$ contains every coherent system of roots of unity if and only if $F$ has \vfk, as required.
\end{proof}

\subsection{Kernel extensions}
\begin{prop}\label{ker ext prop}
Let $F$ be a partial E-field with full kernel, and let $A$ be a subgroup of $\ga(F)$ such that $A \cap D(F) = \ker(F)$, and $A \models \Th\tuple{\Z;+}$. Then there is a unique partial E-field extension $F'$ of $F$ with the same underlying field, such that $D(F')$ is spanned by $D(F) \cup A$ and $\ker(F') = A$.
\end{prop}
\begin{proof}
Any element of $D(F')$ is of the form $x = a/m + b$ with $a \in A$, $m \in \N^+$, and $b \in D(F)$. The exponential map $\exp_F$ restricts to a surjection $\ra{\frac{1}{m}\ker(F)}{}{\sqrt[m]{1}}$.  Since $A$ is a model of $\Th\tuple{\Z;+}$ and a pure extension of $\ker(F)$, this surjection extends uniquely to a homomorphism $\ra{\frac{1}{m}A}{\theta}{\sqrt[m]{1}}$. 

Now $\exp_{F'}(x)$ must satisfy 
\[\exp_{F'}(x) = \exp_{F'}(a/m)\exp_{F'}(b) = \theta(a/m) \exp_F(b)\]
so $\exp_{F'}$ is unique, and it is well-defined because $A \cap D(F) = \ker(F)$.
\end{proof}

The special case of this proposition which we actually use in this paper is as follows.
\begin{cor}\label{Fker cor}
Suppose $F_0$ is a partial E-field satisfying axioms 2$'$ and 3$'$. Let $R \models \Th\tuple{\Z;+,\cdot}$ be a ring extension of $Z(F_0)$ such that $Z(F_0)$ is relatively algebraically closed in $R$. Then, up to isomorphism, there is a unique partial E-field extension $F'$ of $F_0$ such that $Z(F') \iso_{Z(F_0)} R$, $F'$ is generated as a partial E-field by $F_0 \cup Z(F')$, and $\algindep{F_0}{\ker(F')}{\ker(F_0)}$.
\end{cor}
\begin{proof}
Let $F$ be the field extension of $F_0$ generated by a copy of $R$ which is algebraically disjoint from $F_0$ over (the relatively algebraic closure of) $Z(F_0)$. It is unique up to isomorphism. Now apply Proposition~\ref{ker ext prop} with $A = \tau R$, where $\tau$ is the kernel generator.
\end{proof}

\subsection{Strong and semistrong extensions}

Let $A$ be a subset of a partial E-field $F$. We use the notation $\gen{A}_F$ for the smallest partial E-subfield of $F$ containing $A$, and $\gen{A}^{ELA}_F$ for the smallest partial E-subfield $K$ of $F$ containing $A$ which is closed under exponentiation, taking logarithms (that is, if $b \in K$, $a\in F$ and $\exp_F(a) = b$ then $a \in K$), and is relatively algebraically closed in $F$. In particular, if $F$ is an ELA-field then $\gen{A}^{ELA}_F$ is the smallest ELA-subfield of $F$ containing $A \cup \ker(F)$.

Recall that the Schanuel property for $F$ states that the predimension function $\delta(\abar) = \td(\abar,e^\abar) - \ldim_\Q(\abar)$ is non-negative for all tuples $\abar$ from $F$. If $A$ is a subset of a partial (or total) exponential field $F$ and $\abar$ is a finite tuple from $D(F)$, we also consider the \emph{relative predimension function}
\[\delta(\abar/A) = \td(\abar,\exp(\abar) / A,\exp(A)) - \ldim_\Q(\abar/A).\]

\begin{defn}
We say that $A$ is \emph{strong} in $F$, and write $A \strong F$, if and only if for every finite tuple $\abar$ from $D(F)$, $\delta(\abar/A) \ge 0$. We define a partial E-subfield $F_0$ of $F$ to be \emph{strong}, and write $F_0\strong F$, if and only if for every finite tuple $\abar$ from $F$, $\delta(\abar/D(F_0)) \ge 0$. 
\end{defn}
So if $A \subs D(F)$ and $F_0 = \gen{A}_F$, then $F_0 \strong F$ as a partial E-subfield precisely when $A \strong F$ as a subset.

The predimension function $\Delta_F$ satisfies $\Delta_F(\abar) = \delta(\abar/\ker(F))$, and hence the Schanuel property over the kernel is equivalent to the assertion that $\ker(F)$ is strong in $F$.

The notion of strong extensions was extremely important in the paper \cite{FPEF}, but extensions of exponential fields where the kernel extends will not generally be strong. However, much of the technology of strong extensions can be adapted to this broader setting. We define \emph{semistrong} extensions to play the central role.

\begin{defn}
Let $F_0 \subs F$ be an extension of partial E-fields, and $\abar$ a finite tuple from $F$. We define the predimension function
\[ \Delta_{F}(\abar/F_0) = \td(\abar, \exp_{F}(\abar)/ F_0 \cup \ker(F)) - \ldim_\Q(\abar \cap D(F) / D(F_0)\cup \ker(F)) \]
where $\exp_{F}(\abar) = \class{\exp_F(a)}{a \in \abar \cap D(F)}$.

If $A$ is a subset of $F$,  we define $\Delta_{F}(\abar/A) \leteq \Delta_{F}(\abar/F_0)$ where $F_0 = \gen{\ker(F)\cup A}_{F}$, the partial E-subfield of $F$ generated by $\ker(F) \cup A$.

\end{defn}

The following properties are almost immediate from the definition.
\begin{lemma}[Basic properties of the predimension function] \ 
\begin{enumerate}[(a)]
\item $\Delta_F(\abar/A) = \td(\abar,\exp(\abar)/A \cup \ker(F)) - \mrk(\exp_F(\abar)/I(A))$
\item Addition property: for finite tuples $\abar$ and $\bbar$, 
\[\Delta_F(\abar \cup \bbar/A) = \Delta_F(\abar/A \cup \bbar) + \Delta_F(\bbar/A).\]

\item Suppose $F_1 \subs F_2 \subs F_3$ are partial E-fields, that $\abar \in D(F_2)$ , and that $\algindep{F_2}{\ker(F_3)}{\ker(F_2)}$, which means that every finite tuple $\bbar \in F_2$ satisfies $\td(\bbar/\ker(F_2)\cup \ker(F_3)) = \td(\bbar/\ker(F_2))$. Then $\Delta_{F_3}(\abar/F_1) = \Delta_{F_2}(\abar/F_1)$. 
\end{enumerate}\qed
\end{lemma}
Note that the function $\Delta_F$ does generally depend on $F$, but (c) above gives an important situation where there is no dependence.

\begin{defn}
  Let $\theta: F_0 \into F$ be an extension of partial E-fields. We say that the extension is \emph{semi-strong} and write $F_0 \rInto^{\sstrong} F$ or just $F_0 \sstrong F$ if and only if for every $\abar \in F$, $\Delta_F(\abar/\theta(F_0)) \ge 0$, and $\algindep{F_0}{\ker(F)}{\ker(F_0)}$.
\end{defn}
We will make use of the following basic properties of semistrong extensions, which again are almost immediate from the definition. From the last property in the list, we see that they also apply to strong extensions where the kernel does not extend.
\begin{lemma}[Basic properties of semi-strong extensions]\label{ss basic props} \
\begin{enumerate}[(a)]
\item For any partial E-field $F$, $F \sstrong F$
\item If $F_1 \sstrong F_2$ and $F_2 \sstrong F_3$ then $F_1 \sstrong F_3$.
\item $A \sstrong F$ if and only if for every finite tuple $\abar$ from $F$, $A \sstrong \gen{A \cup \abar}_F$.
\item If $F_1 \subs F_2$, $F_1 \sstrong F_3$ and $F_2 \sstrong F_3$ then $F_1 \sstrong F_2$.
\item\label{ss union lemma} Suppose $F_1 \subs F_2 \subs \cdots \subs F_\alpha \subs \cdots $ is a chain of partial E-subfields of some partial E-field $F$, and each $F_\alpha \sstrong F$. Then $\bigcup_{\alpha} F_\alpha \sstrong F$.
\item If $A \sstrong F$ and $\ker(F) \subs A$ then $A \strong F$.
\end{enumerate}\qed
\end{lemma}

We need one further property, the existence of a smallest semistrong extension containing a given tuple.
\begin{prop}\label{hull}
Let $F$ be an ELA-field, let $A \sstrong F$ and let $\bbar$ be a finite tuple from $F$. Then there is a smallest partial E-subfield $B$ of $F$ such that $D(A) \cup \ker(F) \cup \bbar \subs D(B)$ and $B \sstrong F$. That is, if $C$ is any other such partial E-subfield then $D(B) \subs D(C)$.
\end{prop}
\begin{proof}
Exactly the same as for Lemma~7.2 of \cite{FPEF}, using the predimension $\Delta_F$ in place of $\delta$.
\end{proof}

\subsection{The free ELA-extension}
In \cite{FPEF}, it is shown that certain countable partial E-fields have well-defined free ELA-extensions. (Recall that these are algebraically closed extension fields whose exponential map is surjective.) Here we prove the same conclusion replacing the countability and the other hypotheses with the single assumption of \vfk.

\begin{prop}\label{vfk prop}
  Suppose that $F$ is a partial E-field with \vfk, which is algebraic over its graph of exponentiation, that is, $F$ is algebraic over $D(F) \cup I(F)$. Suppose also that $F \strong K$ and $F \strong M$ are two kernel-preserving strong extensions into ELA-fields. Then $\gen{F}_K^{ELA} \iso_F \gen{F}_M^{ELA}$. In particular:
\begin{enumerate}
\item[(1)] If $F$ is a partial E-field with \vfk\ then the free ELA-closure of $F$, $F^{ELA}$, is well-defined.
\item[(2)] If $F$ is a partial E-field with \vfk\ and $F \strong K$ is a kernel-preserving extension to an ELA-field then $\gen{F}_K^{ELA} \iso_F F^{ELA}$.
\end{enumerate}
\end{prop}

\begin{proof}
First note that (1) and (2) follow at once from the main statement of the proposition, because there does exist some strong ELA-field extension of $F$ with no new kernel elements  \cite[Construction~2.13]{FPEF}.

List $\gen{F}_K^{ELA}$ as $(s_{\alpha+1})_{\alpha<\lambda}$ for $\lambda = |\gen{F}^{ELA}_K|$, such that for each $\alpha < \lambda$, either
\begin{enumerate}[i)]
 \item $s_{\alpha+1}$ is algebraic over $F \cup \class{s_\beta}{\beta\le \alpha}$; or
 \item $s_{\alpha+1} = \exp_K(a)$ for some $a \in F \cup \class{s_\beta}{\beta\le \alpha}$; or
 \item $\exp_K(s_{\alpha+1}) = b$ for some $b\in F \cup \class{s_\beta}{\beta\le \alpha}$.
\end{enumerate}
This is possible by the definition of $\gen{F}_K^{ELA}$. 

We will inductively construct chains of partial E-subfields \[F= K_0 \subs K_1 \subs K_2 \subs \cdots \subs K_\alpha \subs \cdots\]
 of $K$ and
\[F= M_0 \subs M_1 \subs M_2 \subs \cdots \subs M_\alpha \subs \cdots\]
of $M$, and nested isomorphisms $\theta_\alpha: K_\alpha \to M_\alpha$ such that for each $\alpha < \lambda$ we have $s_\alpha \in K_\alpha$, $K_\alpha \strong K$, $M_\alpha \strong M$, and $K_\alpha$ is algebraic over its graph of exponentiation.

We start by taking $\theta_0$ to be the identity map on $F$. Now suppose we have $K_\beta$, $M_\beta$, and $\theta_\beta$ for all $\beta <\alpha$. If $\alpha$ is a limit, take unions. All the inductive hypotheses are easily seen to hold. Now consider $\alpha = \gamma +1$.

\medskip
\paragraph{\textbf{Case 1)}} $s_{\gamma+1}$ is algebraic over $K_\gamma$ (including the case where $s_{\gamma+1} \in K_\gamma$). Let $p(X)$ be the minimal polynomial of $s_{\gamma+1}$ over $K_\gamma$. The image $p^\theta$ of $p$ is an irreducible polynomial over $M_\gamma$, so let $t$ be any root of $p^\theta$ in $M$. Let $K_{\gamma+1} = K_\gamma(s_{\gamma+1})$, $M_{\gamma+1} = M_\gamma(t)$, and let $\theta_{\gamma+1}$ be the unique field isomorphism extending $\theta_\gamma$ and sending $s_{\gamma+1}$ to $t$. We make $K_{\gamma+1}$ and $M_{\gamma+1}$ into partial exponential fields by taking the graph of exponentiation to be the graph of $\exp_K$ or $\exp_M$ intersected with $K_{\gamma+1}^2$ or $M_{\gamma+1}^2$ respectively. Suppose that $(a,\exp_K(a)) \in K_{\gamma+1}^2$. Since $K_\gamma \strong K$, we have $\td(a,\exp_K(a)/K_\gamma) - \ldim_\Q(a/D(K_\gamma)) \ge 0$. But $K_{\gamma+1}$ is an algebraic extension of $K_\gamma$, so it follows that $\ldim_\Q(a/D(K_\gamma)) = 0$, that is, that $a \in D(K_\gamma)$. Hence $D(K_{\gamma+1}) = D(K_\gamma)$. The same argument shows that $D(M_{\gamma+1}) = D(M_\gamma)$. Now if $\xbar$ is any tuple from $K$, we have $\delta(\xbar/D(K_{\gamma+1})) = \delta(\xbar/D(K_\gamma)) \ge 0$, and hence $K_{\gamma+1} \strong K$, and similarly $M_{\gamma+1} \strong M$. Since we have just added an algebraic element, we have $K_{\gamma+1}$ algebraic over its graph of exponentiation.

\medskip
\paragraph{\textbf{Case 2)}} $s_{\gamma+1}$ is transcendental over $K_\gamma$ and $s_{\gamma+1} = \exp_K(a)$ for some $a \in K_\gamma$. Let $K_{\gamma+1} = K_\gamma(\sqrt{s_{\gamma+1}})$, by which we mean adjoin all $m^\mathrm{th}$ roots of $s_{\gamma+1}$ for all $m \in \N^+$,  and let $M_{\gamma+1} = M_\gamma(\sqrt{\exp_M(\theta_\gamma(a))})$. Extend $\theta_\gamma$ by defining $\theta_{\gamma+1}(\exp_K(a/m)) = \exp_M(\theta_\gamma(a)/m)$, and extending to a field isomorphism. This is possible because $s_{\gamma+1}$ is transcendental over $K_\gamma$ and $\exp_M(\theta_\gamma(a))$ is transcendental over $M_\gamma$ (the latter because $M_\gamma \strong M$), and so there is a unique isomorphism type of a coherent system of roots of $s_{\gamma+1}$ over $K_\gamma$, and of $\exp_M(\theta_\gamma(a))$ over $M_\gamma$. Then $\td(K_{\gamma+1}/K_\gamma) = 1$, $a \in D(K_{\gamma+1}) \minus D(K_\gamma)$, and $K_\gamma \strong K$, so $D(K_{\gamma+1})$ is spanned by $D(K_\gamma)$ and $a$. Similarly, $D(M_{\gamma+1})$ is spanned by $\theta_\gamma(a)$ over $D(M_\gamma)$, so $\theta_{\gamma+1}$ is an isomorphism of partial E-fields.

Now if $\xbar$ is any tuple from $K$, we have 
\[\delta(\xbar/D(K_{\gamma+1})) = \delta(\xbar,a/D(K_\gamma)) - \delta(a/D(K_\gamma)) = \delta(\xbar,a/D(K_\gamma)) - 0 \ge 0\]
as $K_\gamma \strong K$, so $K_{\gamma+1} \strong K$. The same argument shows that $M_{\gamma+1} \strong M$. Clearly $K_{\gamma+1}$ is algebraic over its graph of exponentiation.

\medskip
\paragraph{\textbf{Case 3)}} Suppose we are not in case 1) or 2). Then $s_{\gamma+1}$ is transcendental over $K_\gamma$, not of the form $\exp_K(a)$ for any $a \in K_\gamma$, but $\exp_K(s_{\gamma+1}) = b$ for some $b \in K_\gamma$ by the choice of $s_{\gamma+1}$. We may assume that $\sqrt{b} \subs K_\gamma$. If not, apply case 1) $\omega$ times to add the $m^\mathrm{th}$ roots of $b$ for each $m \in \N^+$.

 Let $t \in M$ be such that $\exp_M(t) = \theta_\gamma(b)$. Let $K_{\gamma+1} = K_\gamma(s_{\gamma+1})$ and $M_{\gamma+1} = M_\gamma(t)$. Now for $m\in\N^+$, let 
 \[c_m = \theta_\gamma(\exp_K(s_{\gamma+1}/m)) / \exp_M(t/m).\]
 Then $(c_m)_{m\in\N^+}$ is a coherent system of roots of unity, and $F$ has \vfk\ so, by Lemma~\ref{vfk lemma}, there is $a \in F$ such that $\exp_F(a/m) = c_m$ for each $m\in\N^+$. Extend $\theta_\gamma$ by defining  $\theta_{\gamma+1}(s_{\gamma+1}) = t+a$ and $\theta_{\gamma+1}(\exp_K(s_{\gamma+1}/m)) = \exp_M(\theta_{\gamma+1}(s_{\gamma+1})/m)$ and extending to a field isomorphism. This is possible because $s_{\gamma+1}$ is transcendental over $K_\gamma$ and (since $M_\gamma \strong M$) $t+a$ is transcendental over $M_\gamma$. Then by construction, $\theta_{\gamma+1}$ is an isomorphism of partial E-fields, and as in Case 2 above, we have $K_{\gamma+1} \strong K$, $M_{\gamma+1} \strong M$, and $K_{\gamma+1}$  is algebraic over its graph of exponentiation.

\medskip
\paragraph{\textbf{Conclusion}}
That completes the induction. Let $K_\lambda = \bigcup_{\alpha <\lambda} K_\alpha$. Then $K_\lambda = \gen{F}_K^{ELA}$ because $K_\lambda$ is an ELA-subfield of $K$ containing $F$ and is the smallest such because at each stage we add only elements of $K$ which must lie in every ELA-subfield of $K$ containing $F$. The union of the maps $\theta_\alpha$ gives an embedding of $K_\lambda$ into $M$, and, for the same reason, the image must be $\gen{F}_M^{ELA}$. Hence $\gen{F}_K^{ELA} \iso_F \gen{F}_M^{ELA}$ as required.
\end{proof}

\subsection{Free ELA-fields and semi-strong extensions}

The previous proposition shows that when $F$ is a partial E-subfield of an ELA-field $K$ then the isomorphism type of the ELA-subfield $\gen{F}_K^{ELA}$ it generates is determined by the isomorphism type of $F$, given suitable hypotheses including the kernel not extending. Next we show that even when the kernel does extend, we still have some control.

\begin{prop}\label{ss ELA embedding}
Suppose that $F$ is a partial E-field with full kernel which is algebraic over its graph of exponentiation. Suppose also that $K$ is an ELA-field with \vfk, and $\theta: F \sstrong K$ is a semi-strong embedding. Then $\theta$ extends to a semi-strong embedding $\theta^*: F^{ELA} \sstrong K$. Furthermore, if $\theta(\ker(F)) = \ker(K)$ then $\theta^*(F^{ELA}) = \gen{\theta(F)}^{ELA}_K$.
\end{prop}
Note that if $\theta(\ker(F)) \neq \ker(K)$ then the image of $F^{ELA}$ in $K$ will not be uniquely determined by $\theta$. The proof of the Proposition is similar to that of Proposition~\ref{vfk prop}.

\begin{proof}
List $F^{ELA}$ as $(s_{\alpha+1})_{\alpha<\lambda}$ for $\lambda = |F^{ELA}|$, such that for each $\alpha < \lambda$, either
\begin{enumerate}[i)]
 \item $s_{\alpha+1}$ is algebraic over $F \cup \class{s_\beta}{\beta\le \alpha}$; or
 \item $s_{\alpha+1} = \exp_K(a)$ for some $a \in F \cup \class{s_\beta}{\beta\le \alpha}$; or
 \item $\exp_K(s_{\alpha+1}) = b$ for some $b\in F \cup \class{s_\beta}{\beta\le \alpha}$.
\end{enumerate}

We will inductively construct a chain of partial E-subfields \[F= F_0 \subs F_1 \subs F_2 \subs \cdots \subs F_\alpha \subs \cdots\]
 of $F^{ELA}$ and nested semi-strong embeddings $\theta_\alpha: F_\alpha \rInto^{\sstrong} K$ such that for each $\alpha < \lambda$ we have $s_\alpha \in F_\alpha$ and $F_\alpha \strong F^{ELA}$.
 
We start by taking $\theta_0 = \theta$. Now suppose we have $F_\beta$ and $\theta_\beta$ for all $\beta <\alpha$. If $\alpha$ is a limit, take unions. By Lemma~\ref{ss basic props}(\ref{ss union lemma}), $F_\alpha \strong F^{ELA}$ and $\theta_\alpha(F_\alpha) \sstrong K$. Now consider $\alpha = \gamma +1$.

\medskip
\paragraph{\textbf{Case 1)}} $s_{\gamma+1}$ is algebraic over $F_\gamma$ (including the case where $s_{\gamma+1} \in F_\gamma$). Let $p(X)$ be the minimal polynomial of $s_{\gamma+1}$ over $F_\gamma$. The image $p^\theta$ of $p$ is an irreducible polynomial over $\theta_\gamma(F_\gamma)$, so let $t$ be any root of $p^\theta$ in $K$. Let $F_{\gamma+1} = F_\gamma(s_{\gamma+1})$ and let $\theta_{\gamma+1}$ be the unique field embedding which extends $\theta_\gamma$ and sends $s_{\gamma+1}$ to $t$. We make $F_{\gamma+1}$ into a partial exponential field by taking the graph of exponentiation to be the graph of $\exp_{F^{ELA}}$ intersected with $F_{\gamma+1}^2$. 

Using the fact that $F_\gamma \strong F^{ELA}$, it is easy to see that $D(F_{\gamma+1}) = D(F_\gamma)$. Similarly, since $\theta_\alpha(F_\alpha) \sstrong K$, the graph of $\exp_K$ intersected with $\theta_{\gamma+1}(F_{\gamma+1})^2$ is the same as its intersection with $\theta_{\gamma}(F_{\gamma})^2$, so $\theta_{\gamma+1}$ is an embedding of partial E-fields. By the same argument, $F_{\gamma+1} \strong F^{ELA}$ and $\theta_{\gamma+1}(F_{\gamma+1}) \sstrong K$.

\medskip
\paragraph{\textbf{Case 2)}} $s_{\gamma+1}$ is transcendental over $F_\gamma$ and $s_{\gamma+1} = \exp_F(a)$ for some $a \in F_\gamma$. Let $F_{\gamma+1} = F_\gamma(\sqrt{s_{\gamma+1}})$. 

Since $F_\gamma \strong F^{ELA}$, we see that $s_{\gamma+1}$ is transcendental over $F_\gamma$. Since $\theta_\gamma(F_\gamma) \sstrong K$, we see that $\exp_K(\theta_\gamma(a))$ is transcendental over $\theta_\gamma(F_\gamma) \cup \ker(K)$. Thus there is a unique isomorphism type of a coherent system of roots of $s_{\gamma+1}$ over $F_\gamma$, and of $\exp_K(\theta_\gamma(a))$ over $\theta_\gamma(F_\gamma) \cup \ker(K)$, so we may define $\theta_{\gamma+1}(\exp(a/m)) = \exp_K(\theta_\gamma(a)/m)$, and extend uniquely to a field embedding $\ra{F_{\gamma+1}}{\theta_{\gamma+1}}{K}$.

Then $\td(F_{\gamma+1}/F_\gamma) = 1$, $a \in D(F_{\gamma+1}) \minus D(F_\gamma)$, and $F_\gamma \strong F^{ELA}$, so $D(F_{\gamma+1})$ is spanned by $D(F_\gamma) \cup \{a\}$. Similarly, $D(\theta_{\gamma+1}(F_{\gamma+1}))$ is spanned by $D(\theta_{\gamma+1}(F_{\gamma+1})) \cup \{\theta_\gamma(a)\}$, so $\theta_{\gamma+1}$ is an embedding of partial E-fields.

Now if $\xbar$ is any tuple from $F^{ELA}$, we have 
\[\delta(\xbar/F_{\gamma+1}) = \delta(\xbar,a/F_\gamma) - \delta(a/F_\gamma) = \delta(\xbar,a/F_\gamma) - 0 \ge 0\]
because $F_\gamma \strong F^{ELA}$, and so $F_{\gamma+1} \strong F^{ELA}$. A similar argument replacing $\delta$ by $\Delta_K$ shows that $\theta_{\gamma+1}(F_{\gamma+1}) \sstrong K$.

\medskip
\paragraph{\textbf{Case 3)}} Suppose we are not in case 1) or 2). Then $s_{\gamma+1}$ is transcendental over $F_\gamma$ and not of the form $\exp_{F^{ELA}}(a)$ for any $a \in F_\gamma$, but $\exp_{F^{ELA}}(s_{\gamma+1}) = b$ for some $b \in F_\gamma$. As in the proof of Proposition~\ref{vfk prop}, we may assume $\sqrt b \subs F_\gamma$ by applying case 1) $\omega$ times if necessary. Let $F_{\gamma+1} = F_\gamma(s_{\gamma+1})$.

Consider the sequence $(a_m)_{m\in\N^+} = (\exp_{F^{ELA}}(s_{\gamma+1}/m))_{m \in \N^+}$ of roots of $b$. Choose a coherent sequence $(c_m)_{m\in\N^+}$ of roots of $\theta_\gamma(b)$ in $K$ such that $c_m = \theta_\gamma(a_m)$ for any $m$ with $a_m \in F_\gamma$.

Since $K$ has \vfk, there is $t \in K$ such that $\exp_K(t/m) = c_m$, for each $m \in \N^+$. As $\theta_\gamma(F_\gamma) \sstrong K$, $t$ is transcendental over $\theta_\gamma(F_\gamma)\cup\ker(K)$. Thus we can extend $\theta_\gamma$ by defining  $\theta_{\gamma+1}(s_{\gamma+1}) = t$ and $\theta_{\gamma+1}(a_m) = c_m$ for each $m \in \N^+$ and extending to a field embedding. As in case 2), $\theta_{\gamma+1}$ is an embedding of partial E-fields, $F_{\gamma+1} \strong F^{ELA}$, and $\theta_{\gamma+1}(F_{\gamma+1}) \sstrong K$.

\medskip
\paragraph{\textbf{Conclusion}}
That completes the induction. Let $\theta^* = \bigcup_{\alpha < \lambda} \theta_\alpha$. By Lemma~\ref{ss basic props}(\ref{ss union lemma}), $\theta^*$ is a semi-strong embedding of $F^{ELA}$ into $K$. 

Now suppose $\theta(\ker(F)) = \ker(K)$. Then for each $b \in \theta^*(F^{ELA})$ and each $a \in K$ such that $\exp_K(a) = b$, we have $a \in \theta^*(F^{ELA})$. So in this case  $\theta^*(F^{ELA}) = \gen{\theta(F)}^{ELA}_K$ as required.

\end{proof}

\subsection{Exponential algebraicity and exponential transcendence}\label{exp trans section}

In any (partial) exponential field $F$ there is a pregeometry called exponential algebraic closure, which we write $\ecl^F$. We give a quick account of its definition. Details can be found in \cite{EAEF}. An exponential polynomial (without iterations of exponentiation) is a function of the form $f(\Xbar) = p(\Xbar,e^\Xbar)$ where $p \in F[X_1,\ldots,X_n,Y_1,\ldots,Y_n]$ is a polynomial. We can extend the formal differentiation of polynomials to exponential polynomials in a unique way such that $\frac{\partial e^X}{\partial X} = e^X$.

  A \emph{Khovanskii system} (of equations and
  inequations) consists of, for some $n \in \N$, exponential
  polynomials $f_1,\ldots,f_n$ with equations
  \[f_i(x_1,\ldots,x_n) = 0 \quad \mbox{for } i=1,\ldots,n\] and the
  inequation
\[\begin{vmatrix}
  \frac{\partial f_1}{\partial X_1} & \cdots &\frac{\partial
    f_1}{\partial X_n}\\
  \vdots & \ddots & \vdots \\
  \frac{\partial f_n}{\partial X_1} & \cdots &\frac{\partial
    f_n}{\partial X_n} \end{vmatrix} (x_1,\ldots,x_n) \neq 0.\]
where the differentiation here is the formal differentiation of exponential polynomials. We say $n$ is the \emph{width} of the Khovanskii system.

  For any subset $C$ of $F$, we define $a \in \ecl^F(C)$ if and only if
  there are $n \in \N$, $a_1,\ldots,a_n \in F$, and exponential polynomials $f_1,\ldots,f_n$
  with coefficients from $\Q(C)$ such that $a = a_1$ and
  $(a_1,\ldots,a_n)$ is a solution to the Khovanskii system given by
  the $f_i$.

  We say that $\ecl^F(C)$ is the \emph{exponential algebraic closure} of $C$ in $F$. If $a \in \ecl^F(C)$ we say that $a$ is \emph{exponentially algebraic} over $C$ in $F$, and otherwise that it is \emph{exponentially transcendental} over $C$ in $F$.

\begin{fact}[{\cite[Theorem~1.1]{EAEF}}]
Exponential algebraic closure $\ecl^F$ is a pregeometry in any partial exponential field $F$.
\end{fact}
The dimension notion from the pregeometry $\ecl^F$ is called \emph{exponential transcendence degree} and is denoted by $\etd^F$, or just $\etd$.

The above definition of $\ecl^F$ makes sense in any partial exponential field. A different definition is used in \cite{Zilber05peACF0}, which makes sense only in partial exponential fields with the Schanuel property, that is, where axiom 2 holds. However, the two definitions agree in that case by \cite[Theorem~1.3]{EAEF}, which also tells us the following.
\begin{fact}
Let $F \in \ECF$, $A \sstrong F$ and $\bbar$ be a finite tuple  from $F$. Then
\[\etd^F(\bbar/A) = \min \class{\Delta_F(\bbar \cup \cbar / A)}{\cbar \mbox{ is a finite tuple from } F}\]
\end{fact}
\begin{proof}
If $A \sstrong F$ then $A \cup \ker(F) \strong F$. Since $\ker(F) \subs \ecl^F(\emptyset)$ we have $\etd^F(\bbar/A) = \etd^F(\bbar/A \cup \ker(F))$. Now \cite[Theorem~1.3]{EAEF} says that 
\[\etd^F(\bbar/A \cup \ker(F)) = \min \class{\delta(\bbar \cup \cbar / A \cup \ker(F))}{\cbar \mbox{ is a finite tuple from } F}\]
but $\delta(\bbar \cup \cbar / A \cup \ker(F)) = \Delta_F(\bbar \cup \cbar / A)$ so we are done.
\end{proof}

\subsection{Strong kernel-preserving extensions}\label{classification of strong extensions}

Our study of the class \ECF\ of exponential fields follows the common practice in model theory of understanding the types of finite tuples. The general pattern is to take a suitably saturated model $M$ and two finite $n$-tuples $\abar$ and $\bbar$ from $M$, and to see what conditions we need to establish a back-and-forth system, or even an automorphism of $M$, which takes $\abar$ to $\bbar$.

Clearly a bare minimum would be that $\abar$ and $\bbar$ generate isomorphic partial E-subfields of $M$. Since the predimension $\Delta_M$ is bounded below by 0, we can extend $\abar$ and $\bbar$ to finite tuples $\abar'$ and $\bbar'$ which generate semistrong partial E-subfields of $M$. It turns out that if we choose $\abar'$ and $\bbar'$ such that their $\Q$-linear dimension over the kernel is minimal then they are unique up to choosing a different generating set for the $\Q$-linear subspace they span over the kernel. So we can reduce to the case that $\abar$ and $\bbar$ are semistrong in $M$. Propositions~\ref{vfk prop} and~\ref{ss ELA embedding} then show that the ELA-subfields of $M$ generated by $\ker(M) \cup \abar$ and $\ker(M) \cup \bbar$ are isomorphic. It remains to show that such an isomorphism can be extended to an automorphism of $M$.

In order to do that, we need to understand the finitely generated extensions of ELA-subfields inside $M$.  The finitely generated, strong, kernel-preserving extensions of ELA-fields are explained and classified in sections 3 and 4 of \cite{FPEF}. We explain the necessary results from there. 

Suppose $F \strong K$ is a kernel-preserving strong extension of ELA-fields with \vfk, and suppose $K$ is generated by a finite tuple $\abar$ over $F$, in the sense that the smallest ELA-subfield of $K$ which contains $F \cup \abar$ is $K$ itself. Let $V$ be the algebraic locus $V = \loc(\abar,e^\abar/F)$. We may assume that $\abar$ is $\Q$-linearly independent over $F$, since taking a maximal linearly independent subset of $\abar$ gives the same extension $K$. In this case, we say that the variety $V$ is  \emph{additively free}. Since the extension is kernel-preserving, it follows that no product of integer powers of the $e^{a_i}$ lies in $F$. In this case, we say that $V$ is \emph{multiplicatively free}.

As before, we write $G$ for $\ga \cross \gm$. Each matrix $M \in \Mat_{n \cross n}(\Z)$ defines a homomorphism
$\ra{G^n}{M}{G^n}$ by acting as a linear map on $\ga^n$ and as a
multiplicative map on $\gm^n$. For any subvariety $W \subs G^n$, we write $M\cdot W$
for its image.  An irreducible subvariety $W$ of $G^n$ is said to be \emph{rotund} (in \cite{Zilber05peACF0} the terminology was \emph{ex.\ normal}) if and only if for every matrix $M \in \Mat_{n \cross n}(\Z)$ we have $\dim M\cdot W \ge \rk M$. The extension $F \strong K$ is strong if and only if $V$ is rotund \cite[Proposition~5.2]{FPEF}.

To specify the type of $\abar$ over $F$ we must also specify the locus of $(\abar/r, e^{\abar/r})$ over $F$, for all integers $r$. In general these loci may not be determined by $V$. Equivalently, the variety $=(rI\cdot)^{-1}V$ may not be irreducible, where $rI$ is the scalar matrix. However, a result of the second author known as the \emph{thumbtack lemma} (see for example \cite[Fact~3.7]{FPEF}) shows that there is some integer $m$ such that, if we replace each generator $a_i$ by $a_i/m$, then all such varieties are irreducible. We say $V$ is \emph{Kummer-generic}.

We have reduced to the situation where the ELA-field extension $K$ of $F$ is generated by a finite tuple $\abar$ and such that $V = \loc(\abar,e^\abar/F)$ is additively and multiplicatively free, rotund and Kummer-generic. In section~3 of \cite{FPEF} it is shown that if $F$ is countable and has standard kernel then $K$ is determined up to isomorphism as an extension of $F$ by $V$. The same holds in our situation.

\begin{prop}
Suppose that $F$ is an ELA-field with \vfk, and that $K$ is a kernel-preserving strong extension which is generated by a finite tuple $\abar$ and such that $V = \loc(\abar,e^\abar/F)$ is additively and multiplicatively free, rotund and Kummer-generic. Then $K$ is determined up to isomorphism as an extension of $F$ by $V$. 
\end{prop}
\begin{proof}
The proof is the same as in \cite[Section~3]{FPEF} except that Proposition~\ref{vfk prop} is used in place of Theorem~2.18 of \cite{FPEF}.
\end{proof}
In this case we write the extension $K$ as $F|V$ (spoken as \emph{$F$ extended by $V$}).

If we have $V \subs G^n$ satisfying all the above conditions and furthermore $V$ has dimension $n$, then the extension $K = F|V$ of $F$ is \emph{exponentially algebraic} in the sense that $\ecl^K(F) = K$. The simplest example of a non-exponentially algebraic extension is to adjoin a single exponentially transcendental element, which corresponds to the case where $n=1$ and $V = G$, so $\dim V = 2$. We can now see that an exponential field $F$ satisfies axiom 4, strong exponential-algebraic closedness, if and only if given any finite tuple $\abar$ from $F$, and any finitely generated strong, kernel-preserving, exponentially algebraic ELA-extension $\gen{\abar}^{ELA}_F \strong K$, there is an embedding of $K$ into $F$ over $\gen{\abar}^{ELA}_F$. In other words, it is existential closedness within the class of strong kernel-preserving exponentially algebraic extensions.

The last thing we need from this recap of exponential algebra is to consider how to amalgamate two extensions. In fact this is easy as they have a unique free amalgam.
\begin{lemma}\label{amalgam}
Let $F$ be an ELA-field with \vfk\ and let $V \subs G^n$, $W \subs G^r$ be two additively and multiplicatively free, irreducible, rotund, Kummer-generic subvarieties, defined over $F$. Then 
  \[(F|V)|W \iso (F|W)|V \iso F|(V\cross W)\]
   as extensions of $F$.
\end{lemma}
\begin{proof}
Exactly the same as Lemma~5.9 of \cite{FPEF}, again using Proposition~\ref{vfk prop} in place of Theorem~2.18 of that paper.
\end{proof}

\section{The class \ECF}\label{algebra section}

Recall from the introduction that \ECF (Exponentially Closed Fields with Strong Kernel) is the class of exponential fields $F$ satisfying the following axioms: 
\begin{description}
 \item[1] $F$ is an ELA-field.
\item[2$'$a] The kernel is a cyclic $Z$-module.
 \item[2$'$b] Every element of the kernel is transcendental over $Z$.
\item[2$'$c] $\tuple{Z;+,\cdot} \models \Th\tuple{\Z;+,\cdot}$
 \item[3$'$] The \SPOK\ (SPOK)
\item[4] Strong exponential-algebraic closedness (SEAC)
\end{description}
We have seen unconditionally that these axioms are \Loo-expressible, so unconditionally \ECF\ is an \Loo-class. Later we will see conditionally on CIT that \ECF\ is an elementary class. In this section we use the exponential algebra we have developed to study the class \ECF, without assuming CIT. We describe the appropriate notion of saturated model (saturated over the kernel), show that such models exist in all cardinalities beyond the continuum, and show that they are unique and homogeneous over their kernels.

\subsection{Saturation over the kernel}

We give a stronger version of SEAC, incorporating some saturation.
\begin{defn}
Let $\lambda$ be an infinite cardinal. An ELA-field $F$ is said to satisfy $\lambda$-SEAC if and only if, whenever $V$ is a rotund, additively and multiplicatively free subvariety of $G^n$ defined over $F$ and of dimension $n$, and $A$ is a subset of $F$ with $|A| < \lambda$, then there is $\xbar$ in $F$ such that $(\xbar,e^\xbar) \in V$ and is generic in $V$ over $A$.
\end{defn}

Note that $\aleph_0$-SEAC is the same as SEAC, and that $\aleph_1$-SEAC is incompatible with the Countable Closure Property. 

\begin{defn}\label{sat over kernel defn}
An ELA-field $F$ is \emph{$\lambda$-saturated over its kernel} if and only if $\etd(F) \ge \lambda$ and $F \models \lambda$-SEAC. $F$ is \emph{saturated over its kernel} if and only if it is $|F|$-saturated over its kernel.
\end{defn}

\begin{prop}\label{sat over kernel}
Let $F \in \ECF$ have \vfk. Then there is a kernel-preserving strong extension $F \strong F'$ with $F' \in \ECF$, $|F'| = |F|$, and such that $F'$ is saturated over its kernel.
\end{prop}
\begin{proof}
From Proposition~\ref{classification of strong extensions}, there are only $|F|$-many finitely generated kernel-preserving, strong ELA-extensions of $F$. The extensions amalgamate freely by Lemma~\ref{amalgam}, so a standard construction gives $F'$ as the union of a chain of length $|F| \cross \omega$.
\end{proof}

\subsection{Uniqueness and homogeneity}

\begin{theorem}\label{saturation over the kernel}
Suppose $F, M \in \ECF$, both with \vfk, and suppose $|F| = |M| = \lambda > 2^{\aleph_0}$ and both $F$ and $M$ satisfy $\lambda$-SEAC. Suppose that we have an isomorphism $\theta_Z: Z(F) \rIso Z(M)$, and that $\etd(F) = \etd(M)$, and we have a bijection $\theta_B: B_F \rIso B_M$ between exponential transcendence bases of $F$ and $M$. Suppose furthermore that $F_{00} \sstrong F$ and $M_{00} \sstrong M$ are semistrong partial E-subfields or the empty set, with $|F_{00}| = |M_{00}| < \lambda$ and that $\theta_{00}: F_{00} \rIso M_{00}$ is an isomorphism which is compatible with $\theta_Z$ and with $\theta_B$ (that is, the functions agree where more than one is defined).

Then there is an isomorphism $\theta: F \rIso M$ extending $\theta_Z \cup \theta_B \cup \theta_{00}$.
\end{theorem}

\begin{proof}
List $F$ as $(a_{\mu+1})_{\mu < \lambda}$ and list $M$ as $(b_{\mu+1})_{\mu < \lambda}$. We construct chains $(F_\mu)_{\mu < \lambda}$ and $(M_\mu)_{\mu < \lambda}$ of ELA-subfields of $F$ and $M$ respectively and isomorphisms $\theta_{\mu}: F_\mu \rIso M_\mu$ such that
\begin{enumerate}[(i)]
\item $F_\mu \sstrong F$ and $M_\mu \sstrong M$;
\item $Z(F_\mu) \elsubs Z(F)$ and $Z(M_\mu) \elsubs Z(M)$;
\item $|F_\mu| < \lambda$;
\item $\theta_\mu$ is compatible with $\theta_Z \cup \theta_B$; and
\item If $\mu$ is a successor ordinal, then $a_\mu \in F_\mu$ and $b_\mu \in M_\mu$.
\end{enumerate}

\paragraph{\textbf{Base step}}

Let $\tau$ be transcendental and let $SK$ (for standard kernel) be the partial E-field with $D(SK) = \Q \tau$, $\exp(\tau/n)$ a primitive $n^\mathrm{th}$ root of unity for each $n \in \N^+$, and $SK$ generated as a field by $\tau$ and the roots of unity. This defines $SK$ uniquely up to isomorphism, and it is easy to see that $SK$ embeds semi-strongly in all members of \ECF, uniquely up to choice of $\pm \tau$. So if $F_{00} = \emptyset$ we may redefine $F_{00} = M_{00} = SK$ and choose an isomorphism $\theta_{00}$. So we reduce to the case where $F_{00} \neq \emptyset$.

We consider $Z(F)$ as a model of $\Th(\Z;+,\cdot)$. By the downwards L\"owenheim-Skolem theorem, we can find $Z(F_0) \elsubs Z(F)$ with $Z(F_{00}) \subs Z(F_0)$ and $|Z(F_0)| < \lambda$. Since $F$ has very full kernel, we can also assume that $Z(F_0)$ is algebraically compact. Take $Z(M_0) = \theta_Z(Z(F_0))$. By Corollary~\ref{Fker cor}, the partial E-subfields $F_0'$ of $F$ and $M_0'$ of $M$ generated by $F_{00} \cup Z(F_0)$ and $M_{00} \cup Z(M_0)$ respectively are isomorphic, and there is an isomorphism $\theta_0'$ extending $\theta_{00}$ and compatible with $\theta_Z$.

By Proposition~\ref{ss ELA embedding}, the isomorphism $\theta_0'$ extends to an isomorphism $\theta_0: F_0 \rIso M_0$ for some $F_0 \sstrong F$, with $F_0 \iso (F_0')^{ELA}$ and $M_0 \sstrong M$, with $M_0 \iso (M_0')^{ELA}$. The relevant clauses from (i)---(v) hold immediately.

\medskip
\paragraph{\textbf{Limit steps}}
If $\mu$ is a limit ordinal, let $F_\mu = \bigcup_{\nu<\mu} F_\nu$,  $M_\mu = \bigcup_{\nu<\mu} M_\nu$, and  $\theta_\mu = \bigcup_{\nu<\mu} \theta_\nu$. The clauses (i)---(v) are preserved.

\medskip
\paragraph{\textbf{Successor steps}}

Suppose $\mu = \nu + 1$. Since $F_\nu \sstrong F$ and $B_F$ is an exponential transcendence base of $F$, there are finite tuples $\abar\subs F$ containing $a_\mu$, $\bbar \subs B_F \minus F_\nu$ and $\cbar \subs \ker(F)$ such that 
\[\td(\abar,e^\abar/F_\nu,\cbar,\bbar,e^\bbar) - \mrk(e^\abar/F_\nu,e^\bbar) = 0.\]
Choose $Z(F_\mu) \elsubs Z(F)$ with $Z(F_\nu) \cup \cbar \subs Z(F_\mu)$, and $|Z(F_\mu)| < \lambda$. Let $Z(M_\mu) = \theta_Z(Z(F_\mu))$.

As in the base step, using Corollary~\ref{Fker cor} and Proposition~\ref{ss ELA embedding}, we can extend $\theta_\nu$ to an isomorphism of ELA-fields $\theta_\mu':F_\mu' \rIso M_\mu'$, with $Z(F_\mu') = Z(F_\mu)$ and $Z(M_\mu') = Z(M_\mu)$.
Let $r = |\bbar|$. The isomorphism type of the partial E-field extension generated by an exponentially transcendental $r$-tuple is uniquely defined, so $\theta_\mu'$ extends to an isomorphism $\gen{F_\mu' \cup \bbar}_F \iso \gen{M_\mu' \cup \theta_B(\bbar)}_M$, compatible with $\theta_B$. Hence by Proposition~\ref{ss ELA embedding} again, we may extend $\theta_\mu'$ to an isomorphism 
\[\theta_\mu'' : F_\mu'' \rIso M_\mu''\]
where $(\gen{F_\mu' \cup \bbar}_F)^{ELA} \iso F_\mu'' \sstrong F$ and $(\gen{M_\mu' \cup \theta_B(\bbar)}_M)^{ELA} \iso M_\mu'' \sstrong M$.

Now we may take $\abar$ to be a finite tuple of minimal size $n$ such that $a_\mu \in \gen{F_\mu'' \cup \abar}_F$ and $\gen{F_\mu'' \cup \abar}_F \sstrong F$. Since $F_\mu''$ is an ELA-subfield of $F$ and $\abar$ is algebraically free from the kernel of $F$ over the kernel of $F_\mu''$, $V \leteq \loc(\abar,e^\abar/F_\mu'')$ is additively and multiplicatively free. Since $F_\mu'' \sstrong F$, we therefore have that $V$ is rotund, and since $\etd(\abar / F_\mu'') = 0$, we have that $\dim V = n$. As a field $F_\mu''$ is algebraically closed, so, as discussed in section~\ref{classification of strong extensions}, replacing $\abar$ by $\abar / m$ for some $m\in \N^+$ if necessary, the partial E-field extension $F_\mu'' \subs \gen{F_\mu'' \cup \abar}_F$ is determined up to isomorphism over $F_\mu''$ by $V$. Now $M \models \lambda$-SEAC, so there is $\dbar \in M$ such that $(\dbar,e^\dbar) \in V^\theta$, generic over $M_\mu''$, where $V^\theta$ is the subvariety of $G^n(M)$ corresponding to $V$ under $\theta_\mu''$.

Then $\theta_\mu''$ extends to an isomorphism $\gen{F_\mu'' \cup \abar}_F \rIso \gen{M_\mu'' \cup \dbar}_M$, and, by Proposition~\ref{ss ELA embedding} again, to $\theta_\mu''' : F_\mu''' \rIso M_\mu'''$ where
\[(\gen{F_\mu'' \cup \abar}_F)^{ELA} \iso F_\mu''' \sstrong F \quad \mbox{ and }\quad  (\gen{M_\mu'' \cup \dbar)}_M)^{ELA} \iso M_\mu''' \sstrong M.\]

Now $F_\mu'''$ and $M_\mu'''$ satisfy (i)---(iv) and $a_\mu \in F_\mu'''$. Repeat the process swapping the roles of $F$ and $M$ and starting with $F_\mu'''$ in place of $M_\nu$ and $M_\mu'''$ in place of $F_\nu$ to get $\theta_\mu: F_\mu \rIso M_\mu$, satisfying (i)---(v).

\medskip
\paragraph{\textbf{Conclusion}}
That completes the induction steps. Now $\theta \leteq  \bigcup_{\mu<\lambda} \theta_\mu$ is an isomorphism $\theta: F \rIso M$, extending $\theta_Z \cup \theta_B \cup \theta_{00}$ as required.
\end{proof}

We believe that, by a more careful back-and-forth analysis, one could remove the condition that $\lambda > 2^{\aleph_0}$. (From very full kernel we know that $\lambda \ge 2^{\aleph_0}$.) However, we do not need the stronger statement that would result.

\subsection{Superstability over the kernel}

Let $R \models \Th(\Z;+,\cdot)$ be algebraically compact, and consider the subclass 
\[\ECF_R = \class{F \in \ECF}{Z(F) = R}.\]

Then, by Proposition~\ref{sat over kernel}, for each cardinal $\lambda \ge |R|$, there is a model in $\ECF_R$  of cardinality $\lambda$ that is saturated over its kernel. Theorem~\ref{saturation over the kernel} tells us that saturation over the kernel is the same thing as saturation within the class $\ECF_R$, and furthermore that these saturated models are unique up to isomorphism and are homogeneous models (at least after adding parameters for all elements of $R$). Thus the class $\ECF_R$ has a monster model $M$ which is a homogeneous structure, and which is $\lambda$-stable (that is, it realises only $\lambda$ types over any set of size $\lambda$) for all $\lambda \ge |R|$. Hence $\ECF_R$ is a \emph{homogeneous class} in the sense of \cite{HySh:629}. Furthermore, by Theorem~1.17 of that paper, the cardinal invariant $\kappa(M)$ is equal to $\aleph_0$, and hence by definition (see Lemma~5.1 of that paper), the classes $\ECF_R$ are superstable. We can say informally that the class \ECF\ is \emph{superstable over the kernel}. There is a general theory of superstability in homogeneous classes, in close parallel to that for first-order theories.

\section{Complete elementary theory}

In this section we show that, if the conjecture CIT is true, then the class \ECF\ is an elementary class, and furthermore its elementary theory is complete.

\subsection{First-order axioms}

In section~\ref{classification of strong extensions}, all the definitions used in the statement of axiom 4 were explained. We now give a first-order axiom scheme which expresses the axiom, at least within the class of models of the other axioms.
\begin{lemma}
The following scheme of axioms is first-order expressible. 
\begin{multline*}
(\forall \bbar)(\exists \xbar)(\forall \mbar \in Z^{n+r})(\forall t \in \ker)(\forall(\bar{w},\ybar) \in V) \\
\Bigg[(\xbar,e^\xbar)\in V \wedge  
\left[\sum_{i=1}^nm_ix_i + \sum_{i=1}^r m_{i+r}b_i = t \to \sum_{i=1}^nm_iw_i + \sum_{i=1}^r m_{i+r}b_i = t \right] \Bigg]
\end{multline*}
where $n,r$ range over $\N^+$ and $V$ ranges over all the irreducible, rotund, multiplicatively free subvarieties of $G^n$ which are defined over $F$ and of dimension $n$.

Furthermore if $F$ is a model of axioms  1, 2$'$ and 3$'$ then it satisfies axiom 4 (\seacness) if and only if it satisfies the scheme.
\end{lemma}

\begin{proof}
In words, the axiom scheme states that for each $\bbar$ there is $(\xbar,e^\xbar) \in V$ with $\xbar$ not satisfying any $Z$-linear dependencies over $\bbar$ and the kernel except those which hold on all of $V$.

It is well-known (part of the \emph{fibre dimension theorem}) that when $V$ varies in a parametric family, the set of parameters such that $V$ is irreducible and of dimension $n$ is first-order definable in the field language. Theorem~3.2 of \cite{Zilber05peACF0} shows that the properties \emph{rotundity} and \emph{multiplicative freeness} of an algebraic variety are also first-order definable just in the field language. So the axiom scheme is first-order expressible.

Now suppose that $V$, $\bbar$, and $\xbar$ are as given. If $V$ happens also to be additively free then $\xbar$ does not satisfy any $\Z$-linear dependencies over $\bbar$ and the kernel , and hence $e^\xbar$ is multiplicatively independent over $e^\bbar$. Thus, using strong kernel and extending $\bbar$ if necessary so that $\bbar \sstrong F$ and $V$ is defined over $\bbar$, we deduce that $(\xbar,e^\xbar)$ has transcendence degree $n$ over $\bbar \cup e^\bbar \cup \ker$, and hence in particular it is generic in $V$ over $\bbar$. So the axiom scheme implies the SEAC property.

For the converse, suppose that $F \in \ECF$, that $V \subs G^n$ is irreducible, rotund, additively and multiplicatively free, and of dimension $n$, and that $\bbar \in F^r$. Extend $\bbar$ such that $\bbar \sstrong F$ and $V$ is defined over $\bbar$. Then the SEAC property gives $(\xbar,e^\xbar) \in V$, generic over $\bbar \cup e^\bbar$. In particular, since $V$ is multiplicatively free, we have $\mrk(e^\xbar/e^\bbar) = n$. So, since $\bbar \sstrong F$, we deduce that $\td(\xbar,e^\xbar/\bbar,e^\bbar,\ker(F)) = n$, and hence $(\xbar,e^\xbar)$ is generic in $V$ over $\bbar \cup \ker(F)$. Now $Z(F) = \tau^{-1} \ker(F)$, so $Z(F)$ is algebraic over $\ker(F)$. So $(\xbar,e^\xbar)$ is generic in $V$ over $Z(F)$ and thus $\xbar$ satisfies no $Z(F)$-linear dependencies over $\bbar \cup \ker(F)$, except those holding on all of $V$. So the SEAC scheme holds on $F$.
\end{proof}

\begin{prop}\label{ECF elementary}
If CIT is true then axioms 1, 2$'$, 3$'$ and 4 are first-order expressible, and so $\ECF$ is an elementary class.
\end{prop}
\begin{proof}
It is clear that axioms 1 and 2$'$ are first-order expressible. By Proposition~\ref{strong kernel is fo}, axiom 3$'$ (the Schanuel property over the kernel) is first-order expressible if CIT is true. Then by the above lemma, axiom 4 is first-order expressible.
\end{proof}

\subsection{Infinite exponential transcendence degree}

\begin{prop}\label{exp trans prop}
Suppose $F \in \ECF$ is $\lambda$-saturated, and $A \subs F$ is an exponential subfield of size less than $\lambda$. Then $F$ contains an element $b$ which is exponentially transcendental over $A$. In particular, the exponential transcendence degree of $F$, $\etd(F)$, is at least $\lambda$.
\end{prop}
 We use CIT only implicitly here. By Proposition~\ref{false CIT}, if CIT is false then no $F \in \ECF$ is $\lambda$-saturated so the proposition is trivially true. 
\begin{proof}
Extend $A$ if necessary such that $A \sstrong F$. For an $n$-tuple $\fbar$ of exponential polynomials (without iterations of exponentiation), consider the formula $\chi_{\fbar}(\xbar)$ given by
 \[\bigwedge_{i=1}^n f_i(x_1,\ldots,x_n) = 0 \wedge \begin{vmatrix}
  \frac{\partial f_1}{\partial X_1} & \cdots &\frac{\partial
    f_1}{\partial X_n}\\
  \vdots & \ddots & \vdots \\
  \frac{\partial f_n}{\partial X_1} & \cdots &\frac{\partial
    f_n}{\partial X_n} \end{vmatrix} (x_1,\ldots,x_n) \neq 0.\]
The type of an exponentially transcendental element $x$ over $A$ is given by all the formulas
\[\neg \exists x_2,\ldots,x_n [\chi_{\fbar}(x, x_2,\ldots,x_n)] \]
where $\fbar$ ranges over all finite lists of exponential polynomials with coefficients from $A$.

Since $F$ is $\lambda$-saturated, we just have to show that this type is finitely satisfied inside $F$. Given a finite set of such formulas, let $N\in \N$ be greater than the width of any of the corresponding Khovanskii systems. Suppose $b_1 \in F$ and, setting $b_{r+1} = \exp(b_r)$ for $r = 1,\ldots,N$, we have $b_{N+1} = b_1$ but $b_1,\ldots b_N$ are algebraically independent over $A \cup \ker(F)$. Such a $b_1$ exists by axiom 4 (\seacness) and saturation. Write $\bbar$ for the tuple $(b_1,\ldots,b_N)$, write $A'$ for the algebraic closure of $A \cup \ker(F)$ and write $B$ for the $\Q$-linear span of $A' \cup \bbar$. Then $\Delta_F(B/A) = 0$, so $B \sstrong F$.

Now suppose that $n < N$ and that $\cbar = (c_1,\ldots, c_n)$ is an $n$-tuple from $B$. Let $C$ be the $\Q$-linear span of $A' \cup \cbar$. We claim that $\Delta_F(C/A) > 0$.

To see this, first we note that we can assume that the $c_i$ are $\Q$-linearly independent over $A'$. Let $d_i = \exp(c_i)$. Then the $d_i$ are Laurent monomials in the $b_i$, so $c_1,\ldots,c_n, d_1,\ldots,d_n \in A'(b_1,\ldots,b_N)$. Since the $b_i$ are algebraically independent over $A'$, we can identify this field with the field $A'(X_1,\ldots,X_N)$ of rational functions in $N$ variables over $A'$ and we identify $b_i$ with the variable $X_i$.

Now we extend $c_1,\ldots,c_n$ to a linear basis $c_1,\ldots,c_n,\ldots,c_N$ for $B$ over $A'$, such that each $c_i$ is a linear polynomial in the $X_i$, with coefficients from $\Q$ except for the constant term which may be from $A'$. So we also have an equality of fields $A'(X_1,\ldots,X_N)= A'(c_1,\ldots,c_N)$. Note that $c_1,\ldots,c_N$ is also a transcendence base of the field over $A'$. 

Suppose for a contradiction that $\td(C, \exp(C)/A')=n$. We will have that each $d_i\in A'(c_1,\ldots,c_n)^\alg$ and also by the above we have $d_i\in A'(c_1,\ldots,c_N)$. Using the algebraic independence of the $c_i$, it follows that $d_i\in A'(c_1,\ldots,c_n)$, that is, $d_i=r_i(c_1...c_n)$, a rational function. Looking more closely, we can see that $r_i$ is a product $\prod_k L_{k,i}$ of linear combinations $L_{k,i}$ of $c_1,\ldots,c_n$ and their inverses.

Now $d_1$ is a monomial in $X_1,\ldots,X_N$, so say $X_1$ occurs in a positive power. Let $c'_1,\ldots,c'_n$ be the results of substituting $X_1=0$. Then for some positive $L_{k,1}$ we have $L_{k,1}(c'_1,\ldots,c'_n)=0$. This proves that $L_{k,1}(c_1,\ldots,c_n)=X_1$. Hence $X_1\in C$. Hence $X_2 \in \exp(C)$, so $X_2$ is present say in $d_2$. Continuing by induction we get $X_1,\ldots,X_N\in C$, which contradicts $n < N$.

Hence $\td(\cbar,\exp(\cbar)/A') > n$, so $\Delta_F(C/A) >0$. By Proposition~\ref{hull}, there is a smallest $\Q$-linear subspace of $F$ containing $A' \cup \{b_1\}$ which is semistrong in $F$. The above calculation shows that it cannot be a proper subspace of $B$ and hence it is $B$ itself.

Now if we have $c_2,\ldots,c_n \in F$ and a Khovanskii system $\chi_\fbar$ such that $F \models \chi_\fbar(b_1,c_2,c_3,\ldots,c_n)$ then $\Delta_F(b_1,c_2,c_3,\ldots,c_n) = 0$. Hence the $\Q$-linear span of $b_1,c_2,\ldots,c_n$ over $A'$ contains $B$, and hence $n \ge N$.

Thus by our choice of $N$, $b_1$ satisfies the chosen finite part of the type of an element which is exponentially transcendental over $A$. Hence, by compactness and $\lambda$-saturation, $F$ contains an element which is exponentially transcendental over $A$.
\end{proof}

\subsection{The completeness proof}
\begin{theorem}\label{completeness theorem}
Assuming CIT, the first-order theory of \ECF, given by axioms 1, 2', 3', and 4, is complete. 
\end{theorem}

The idea of the proof is to show that two saturated models of the same cardinality are isomorphic. Most of the work for the proof was done in section~\ref{algebra section}, but we still have to relate saturation with respect to the first order theory to our notion of saturation over the kernel. Furthermore, to work within ZFC set theory we use special models instead of saturated models.
\begin{proof}
Assume CIT. So \ECF\ is an elementary class by Proposition~\ref{ECF elementary}. Let $\lambda$ be a cardinal such that $\lambda = \sup \class{2^\mu}{\mu < \lambda}$ and $\lambda > 2^{\aleph_0}$, which exists by \cite[5.1.7]{ChangKeisler}. Suppose that $F, M \in \ECF$ are special models of cardinality $\lambda$. That is, we can write $F = \bigcup_{\mu < \lambda} F_\mu$ and $M = \bigcup_{\mu < \lambda} M_\mu$, where $F_\mu$ and $M_\mu$ are $\mu^+$-saturated and form elementary chains. Every complete theory has exactly one special model of cardinality $\lambda$ \cite[5.1.8, 5.1.17]{ChangKeisler}, so it is enough to show that $F \iso M$.

Since $F$ and $M$ are special, and the rings $Z(F)$ and $Z(M)$ are $\emptyset$-definable subsets which model $\Th(\Z;+,\cdot)$, by \cite[5.1.6(v)]{ChangKeisler} they are special models of $\Th(\Z;+,\cdot)$. They are both of cardinality $\lambda$, and hence they are isomorphic. Note in particular that $Z(F)$ is $\aleph_0$-saturated, so $F$ and $M$ have \vfk.

Since $F_\mu$ is $\mu^+$-saturated it has exponential transcendence degree at least $\mu^+$ by Proposition~\ref{exp trans prop}, so there is an exponentially-algebraically independent subset $B \subs F_\mu$ of cardinality $\mu^+$. Being exponentially-algebraically independent is a type-definable property, so $B$ is still exponentially-algebraically independent as a subset of $F$. Thus $\etd(F) \ge \mu^+$ for each $\mu < \lambda$, and so $\etd(F) \ge \lambda$. Since $|F| = \lambda$ we have $\etd(F) = \lambda$. 

Now we show that $F$ has $\lambda$-SEAC. Let $A \subs F$ be a subset with $|A| = \mu < \lambda$, and let $V \subs G^n$ be rotund, and additively and multiplicatively free. By replacing $A$ by a larger subset (of the same cardinality) we may assume that $A$ is semistrong in $F$. Now $F_\mu$ is $\mu^+$-saturated and so satisfies $\mu^+$-SEAC, and hence there are $\bbar_\nu$ for $\nu < \mu^+$ with $(\bbar_\nu,e^{\bbar_\nu}) \in V$, algebraically independent of each other (over the field of definition of $V$). So, in particular, all the $e^{b_{\nu,i}}$ for $\nu < \mu^+$ and $i=1,\ldots,n$ are multiplicatively independent. Since $|A| = \mu < \mu^+$, there must be some $\nu_0$ such that $e^{\bbar_{\nu_0}}$ is multiplicatively independent over $\exp(A)$. But since $A$ is semistrong, $\td(\bbar_{\nu_0},e^{\bbar_{\nu_0}}/A,\exp(A)) \ge \mrk(e^{\bbar_{\nu_0}}/ \exp(A))=n$ and hence $(\bbar_{\nu_0},e^{\bbar_{\nu_0}})$ is generic in $V$ over $A$, as required.

The same arguments show that $\etd(M) = \lambda$ and that $M \models \lambda$-SEAC. Thus, by Theorem~\ref{saturation over the kernel}, there is an isomorphism $F \iso M$.
\end{proof}

\subsection{Exponential-algebraic closedness}

In the second author's paper \cite{Zilber05peACF0}, the \seacness\ property was considered alongside the apparently weaker  \emph{exponential-algebraic closedness} property. We now state this property in a simpler form than in the earlier paper and give several equivalent statements. We then show that, under CIT, it can replace \seacness\ as an axiom for the class \ECF.

An ELA-field $F$ is said to be \emph{exponentially-algebraically closed}, or EAC, if and only if for every $n \in \N^+$ and every rotund subvariety $V$ of $G^n$ defined over $F$ there is $\xbar$ in $F$ such that $(\xbar,e^\xbar) \in V$. The EAC property is first-order axiomatizable because rotundity is definable in the field language \cite[Theorem~3.2(5)]{Zilber05peACF0}. (The notion of rotundity, while tailored to exponentiation, is an algebraic notion not using the exponential structure, so in particular it does not matter that that paper only dealt with exponential fields with standard kernel.)

The next two lemmas summarize some variant but equivalent statements of exponential-algebraic closedness.

\begin{lemma}\label{first EAC lemma}
Let $F$ be an ELA-field of infinite transcendence degree (for example, satisfying the strong kernel axiom). Then $F \models $ EAC if and only if for every $n \in \N^+$ and every rotund subvariety $V$ of $G^n$ defined over $F$ which is
\begin{enumerate}[(i)]
\item irreducible;
\item of dimension $n$;
\item multiplicatively free; and
\item additively free
\end{enumerate}
there is $(\xbar,e^\xbar) \in V$. In particular, if $F \in \ECF$ then $F \models $ EAC.
\end{lemma}
\begin{proof}
The $\Rightarrow$ direction is immediate. For the $\Leftarrow$ direction, we consider $(i)$---$(iv)$ in turn. 

For $(i)$, a reducible $V$ is rotund by definition if and only if some irreducible component of it is rotund. For $(ii)$, if $V \subs G^n$ is rotund but $\dim V > n$, choose $a_1,\ldots,a_{2n} \in F$, algebraically independent over the field of definition of $V$, and let $X$ be the generic hyperplane in $G^n$ given by the equation
\[ \sum_{i=1}^n a_ix_i + a_{n+i} y_i = 1\]
Then for any $M \in \Mat_{n\cross n}(\Z)$, with $\rk(M) < n$, we have
\[\dim M \cdot (V \cap X) = \max(\dim M\cdot V, \rk M)\]
so $V\cap X$ is rotund. (A more detailed version of the same argument is given in the proof of Proposition~2.33, Step~1 in \cite{TEDESV}.) Since $\dim (V \cap X) = \dim V - 1$, we are done by induction on $\dim V - n$.

For $(iii)$, suppose $V$ is not multiplicatively free. If $n=1$ then $V$ has the form $y=c$ for some $c\in \gm(F)$. Since $F$ is an ELA-field, there is $a \in F$ be such that $\exp_F(a) = c$, and $(a,c) \in V$. If $n > 1$, then there are $m_i \in \Z$, not all zero, and $c \in \gm(F)$ such that a point $(\xbar,\ybar)$ of $V$ which is generic over $F$ satisfies $\prod_{i=1}^n y_i^{m_i} = c$. Since $V$ is rotund, $\sum_{i=1}^n m_i x_i$ is transcendental over $F$. Let $V' = V \cap \gen{\sum_{i=1}^n m_i x_i = a}$ where $\exp_F(a) = c$. Then $\dim V' = \dim V - 1$. Assume without loss of generality that $m_n \neq 0$. Consider the $n-1 \times n$ matrix
\[M = \begin{pmatrix} 1 &  & 0 & 0   \\
 & \ddots & & \vdots\\
0 & & 1 & 0 
\end{pmatrix}  \]
and let $V'' = M \cdot V' \subs G^{n-1}$.
Then $\dim V'' = \dim V'$, and indeed for any $M' \in \Mat_{n-1 \times n-1}(\Z)$, we have 
\[\dim M'\cdot V'' = \dim \begin{pmatrix}M' & 0\\ 0 & 0 \end{pmatrix}  \cdot V \ge \rk M'\]
since $V$ is rotund, and hence $V''$ is rotund. By induction on $n$, there are $a_1,\ldots,a_{n-1} \in F$ such that $(a_1,\ldots,a_{n-1},e^{a_1},\ldots,e^{a_{n-1}}) \in V''$, and then taking $a_n$ such that $\sum_{i=1}^n m_i a_i = a$ we have $(\abar,e^\abar) \in V$.

The proof of $(iv)$ is similar to $(iii)$.
\end{proof}

The hypothesis that $F$ has infinite transcendence degree is harmless from our point of view, but recent work by Vincenzo Mantova \cite{Mantova12} suggests that it is unnecessary, because a suitably generic substitute for the generic hyperplane $X$ can be found by other means.

The condition shown in the next lemma to be equivalent to EAC is in fact the original definition of EAC from \cite{Zilber05peACF0}.
\begin{lemma}\label{V - V' lemma}
Let $F$ be an ELA-field of infinite transcendence degree. Then $F \models $ EAC if and only if for every $n \in \N^+$ and every irreducible rotund subvariety $V$ of $G^n$ and every proper (not necessarily irreducible) subvariety $V' \subs V$, there is $(\xbar,e^\xbar) \in V \minus V'$.
\end{lemma}
\begin{proof}
This time the $\Leftarrow$ direction is immediate. We use the classical Rabinovich trick for the $\Rightarrow$ direction. Suppose $V$ is given by polynomial equations $f_i(\xbar,\ybar) = 0$ for $i=1,\ldots,r$, and $V'$ is given by equations $g_i(\xbar,\ybar) = 0$ for $i=1,\ldots,s$. We may assume that $g_1(\xbar,\ybar)$ does not vanish on $V$.

Consider the variety $W \subs G^{n+1}$ given by the equations $f_i(\xbar,\ybar) = 0$ for $i=1,\ldots,r$ and $g_1(\xbar,\ybar) u = 1$, where $u$ is the coordinate for an extra $\ga$, and we also consider a coordinate $v$ for the corresponding $\gm$.

It is enough to show that $W$ is rotund, since then by EAC there are $a_1,\ldots,a_{n+1} \in F$ such that $(a_1,\ldots,a_{n+1}, e^{a_1},\ldots,e^{a_{n+1}}) \in W$, and hence $(a_1,\ldots,a_n, e^{a_1},\ldots,e^{a_n}) \in V$ with $g_1(a_1,\ldots,a_n, e^{a_1},\ldots,e^{a_n}) \neq 0$, so $(a_1,\ldots,a_n, e^{a_1},\ldots,e^{a_n})\notin V'$.

Clearly $\dim W = \dim V + 1$, because the variable $u$ is constrained but the variable $v$ is unconstrained. Now let $M \in \Mat_{n+1\cross n+1}(\Z)$, let $(\alpha,\beta) \in M \cdot W$ be generic over $F$ (with $\alpha \in \ga^{n+1}$ and $\beta \in \gm^{n+1}$), and let $(\abar,c,\bbar,d)$ be generic over $F\cup (\alpha,\beta)$ in the fibre of $(\alpha,\beta)$ of the map $W \rOnto M\cdot W$, where $\abar$ is the tuple of values of the coordinates $\xbar$, $\bbar$ is the values of $\ybar$, $c$ is the value of $u$, and $d$ is the value of $v$.

Say $M = \left( M_1 | M_2 \right)$, where $M_1$ is an $(n+1) \times n$ matrix, and $M_2$ is an $(n+1) \times 1$ matrix.
Then we have $M \cdot (\abar c, \bbar d) = (\alpha,\beta)$, or equivalently 
\[M_1 \abar = \alpha - M_2 c =: \gamma \qquad \mbox{ and } \qquad  \bbar^{M_1} = \beta / d^{M_2} =: \delta.\]
Now $(\alpha,\beta)$ is generic in $M\cdot W$ over $F$, and it follows that $(\gamma,\delta)$ is generic in $M_1 \cdot V$ over $F$, and indeed $(\abar,\bbar)$ is a generic point of the fibre of $(\gamma,\delta)$ of the map $V \rOnto M_1\cdot V$. Thus, by the fibre dimension theorem, and rotundity of $V$ we have
\[\td(\abar,\bbar/F,\gamma,\delta) = \dim V - \dim M_1 \cdot V \le \dim V - \rk M_1.\]

Now $\td(c,d/F,\gamma, \delta, \abar,\bbar) \le 1$, and $\rk M \ge \rk M_1$, so 
\[\td(\abar,\bbar,c,d/F,\gamma,\delta) \le \dim V + 1 - \rk M = \dim W - \rk M\]
but also 
$\td(\abar,\bbar,c,d/F,\gamma,\delta) = \td(\abar,\bbar,c,d/F,\alpha,\beta)$, which is the dimension of a generic fibre of $W \rOnto M\cdot W$ so by the fibre dimension theorem again we have $\dim M \cdot W \ge \rk M$. Hence $W$ is rotund, as required.
\end{proof}

We can now prove Theorem~\ref{EAC theorem}. For convenience, we restate it.
\begin{theorem}
Assuming CIT, the theory \ECF\ is axiomatized by axioms 1, 2$'$, 3$'$, and EAC. 
\end{theorem}
\begin{proof}
Suppose that $F$ satisfies axioms 1, 2$'$, 3$'$, and EAC, and is $\aleph_0$-saturated. (CIT is implicitly needed here, since otherwise there are no $\aleph_0$-saturated models of axiom 3$'$.) Suppose $V \subs G^n$ is rotund, additively and multiplicatively free, and of dimension $n$. Then, since $F$ has strong kernel, for each finite tuple $\bbar \in F$ there is a finite tuple $\cbar$ such that $V$ is defined over $\cbar$, $\cbar \sstrong F$, and $\bbar \subs e^\cbar$. Then by EAC, Lemma~\ref{V - V' lemma}, and $\aleph_0$-saturation, there is $\abar$ such that $(\abar, e^\abar) \in V$ and $e^\abar$ is multiplicatively free over $e^\cbar$. By strong kernel again, $(\abar,e^\abar)$ is generic in $V$ over $\bbar$. Thus, using Lemma~\ref{first EAC lemma}, the two sets of axioms (1, 2$'$, 3$'$, and SEAC, and 1, 2$'$, 3$'$, and EAC), have the same $\aleph_0$-saturated models. But (assuming CIT), they are both lists of first-order axioms, and hence they have the same models.
\end{proof}

\section{Corollaries of CIT}

Our main theorem allows us to see that various statements follow from CIT. In particular, the first below says essentially that $\Z$ is stably embedded in $\B$.
\begin{theorem}
Assume CIT is true, so \ECF\ is a complete first-order theory. Then:
\begin{enumerate}[(1)]
\item Every subset of $Z^n$ which is definable (with parameters) in the theory \ECF\ is also definable with parameters in $Z$, and hence is definable (with parameters) in the theory $\Th(\Z;+,\cdot)$.

\item \ECF\ has quantifier elimination in the language $L'$, which is defined to be $\tuple{+,\cdot,0,1,\exp}$ expanded by predicates for every definable subset of \Z\ and for all $\exists$-formulas. 

\item If $F \subs M$ with $F,M \in \ECF$, then $F \elsubs M$ if and only if $Z(F) \elsubs Z(M)$ and $F \sstrong M$.

\item Let $R_1, R_2 \models \Th(\Z;+,\cdot)$, and suppose $\theta_R: R_1 \into R_2$ is a ring embedding such that $R_1$ is relatively algebraically closed in $R_2$. Then there are $F_1, F_2 \in \ECF$ with $Z(F_i) = R_i$ and a semistrong embedding $\theta: F_1 \sstrong F_2$ extending $\theta_R$. If $\theta_R$ is an elementary embedding then $\theta$ is also an elementary embedding.

\end{enumerate}
\end{theorem}

\begin{proof}
\begin{enumerate}[(1)]
\item By Theorem~\ref{saturation over the kernel}, every automorphism of $Z(F)$ extends to an automorphism of $F$, for a sufficiently saturated model $F$, and this is enough by Lemma~1 of the appendix of \cite{CH99}.

\item Suppose that $F$ is a special model of \ECF, and $\abar$, $\bbar$ are $n$-tuples from $F$ with the same quantifier-free $L'$-type. Then there are finite tuples $\zbar$ from $Z(F)$ and $\cbar$ from $F$ such that $\algindep{\abar}{Z(F)}{\zbar}$, $\cbar \sstrong F$ and $\abar$ lies in the $\Q$-linear span of $\zbar \cup \cbar$.

Since $L'$ contains symbols for all definable subsets of $Z$, and $F$ is $\aleph_0$-saturated, there is $\wbar$ in $Z(F)$ such that $\tp_{L'}(\wbar/\bbar) = \tp_{L'}(\zbar/\abar)$. Let $V = \loc(\cbar,e^\cbar/\zbar,\abar)$. Then, since $L'$ contains predicates for $\exists$-formulas, there is $\dbar$ in $F$ such that $\loc(\dbar,e^\dbar/\wbar,\bbar) = V$. Furthermore, since $F$ is $\aleph_0$-saturated and any failure of semistrongness would be witnessed by an existential formula, we may assume that $\dbar \sstrong F$.

Since $F$ is special, there is an automorphism $\theta_Z$ of $Z(F)$ taking $\zbar$ to $\wbar$. Let $Z_0$ be a countable elementary substructure of $Z(F)$ containing $\zbar$. Let $F_{00}$ be the partial E-subfield of $F$ with $D(F_{00})$ spanned by $Z_0 \cup \cbar$, and let $M_{00}$ be the partial E-subfield of $F$ with $D(M_{00})$ spanned by $\theta_Z(Z_0) \cup \dbar$. Then $\theta_Z\restrict{Z_0}$ extends to an isomorphism $\theta_{00} : F_{00} \rIso M_{00}$. Hence, by Theorem~\ref{saturation over the kernel}, there is an automorphism $\theta$ of $F$ taking $\zbar \cup \cbar$ to $\wbar \cup \dbar$. Now $\abar$ is a $\Q$-linear combination of $\zbar \cup \cbar$ and $\bbar$ is the same linear combination from $\wbar \cup \dbar$, so $\theta(\abar) = \bbar$. Quantifier elimination follows.

\item This follows from (2) and its proof.

\item First build a model $F_1$ such that $Z(F_1) = R_1$, by a standard inductive construction such as given in the proof of \cite[Theorem~6.2]{FPEF}. (The hypotheses of that theorem include countability of $R_1$, but the countability is only used for the uniqueness part of the theorem which we do not need.) Then apply Corollary~\ref{Fker cor} and build another model $F_2$ extending $F_1 \cup R_2$ by the same process. We have $F_1 \sstrong F_2$ by construction, and, by (3), if $\theta_R$ is elementary then $\theta$ will be elementary.
\end{enumerate}

\end{proof}
Part (2) can be rephrased as \B\ being near-model complete except for quantification over \Z. By \cite[Theorem~7.6]{FPEF}, it is not model complete after adding predicates for definable subsets of \Z, so this is best possible. 

Using (3) and (4), and the non-model completeness of $\Th(\Z;+,\cdot)$, we can find  $F, M \in \ECF$ with $F \sstrong M$ but $Z(F) \not\elsubs Z(M)$. Hence there is a non-model completeness of $\B$ arising directly from the non-model completeness of \Z, and hence completely different from that given by the proof in \cite{FPEF}. This second proof fits better with the only non-model completeness proofs known for \Cexp\ \cite[Proposition~1.1]{Marker06}, which do use the integers in an essential way.

\section{Other kernels}

Axiom 2$'$c, the complete theory of $\tuple{\Z;+,\cdot}$, is not used anywhere in this paper except in the proof of Theorem~\ref{completeness theorem} to show that $Z(F) \iso Z(M)$ where $F$ and $M$ are both special models of the same cardinality. The only property used there is that the theory is complete. From axioms 1, 2$'$a, and 2$'$b, we can deduce that $Z$ is an integral domain such that $\tuple{Z;+}$ is elementarily equivalent to $\tuple{\Z;+}$, but that is all that is used. So we could replace 2$'$c by the complete theory of any other integral domain whose additive group is a model of $\Th\tuple{\Z;+}$, and the same conclusions would hold. It would be very interesting to know if there is such a ring with a decidable theory, since together with recursive bounds for all the cases of CIT it would give a decidable complete theory of ELA-fields.

Axioms 2$'$a and 2$'$b, that the kernel is a cyclic $Z$-module and is transcendental over $Z$, are also not used much in the proof. The whole of axiom 2$'$ could be replaced by the complete theory of $\tuple{F;+,\cdot,K}$, where $F$ is an algebraically closed field and $K$ is any subgroup of $\ga(F)$  which is a model of Presburger arithmetic, and which is then taken to be the kernel.

\end{document}